\documentclass{cmslatex}
\usepackage{amsfonts,amsmath,amssymb}
\usepackage{color,tikz, mathdots}
\usetikzlibrary{arrows,decorations.pathmorphing,backgrounds,fit}
\usepackage{wasysym}

\newtheorem{thm}{Theorem}[section]
\newtheorem{lemma1}[thm]{Lemma}
\newtheorem{cor}[thm]{Corollary}
\newtheorem{example}[thm]{Example}
\newtheorem{definition1}[thm]{Definition}
\newtheorem{remark}[thm]{Remark}
\newtheorem{prop}[thm]{Proposition}

\begin{document}

\title{Combinatorial approaches to Hopf bifurcations in systems in interacting elements}

\author{David Angeli\thanks{Imperial College London, Department of Electrical and Electronic Engineering.}
\and Murad Banaji\thanks{University of Portsmouth, Department of Mathematics. Corresponding author: {\tt murad.banaji@port.ac.uk}.}
\and Casian Pantea\thanks{West Virginia University, Department of Mathematics.}
}\maketitle

\begin{abstract}
We describe combinatorial approaches to the question of whether families of real matrices admit pairs of nonreal eigenvalues passing through the imaginary axis. When the matrices arise as Jacobian matrices in the study of dynamical systems, these conditions provide necessary conditions for Hopf bifurcations to occur in parameterised families of such systems. The techniques depend on the spectral properties of additive compound matrices: in particular, we associate with a product of matrices a signed, labelled digraph termed a $\mathrm{DSR}^{[2]}$ graph, which encodes information about the second additive compound of this product. A condition on the cycle structure of this digraph is shown to rule out the possibility of nonreal eigenvalues with positive real part. The techniques developed are applied to systems of interacting elements termed ``interaction networks'', of which networks of chemical reactions are a special case.
\end{abstract}


\section{Introduction}\label{sec:intro}

The material in this paper can be motivated both by abstract questions in linear algebra about the spectra of sets of matrices, and by the study of asymptotic behaviour in dynamical systems. The connection is quite natural: given a sufficiently smooth dynamical system, the structure or spectra of the Jacobian matrices associated with the system may determine certain behaviours, for example the possibility of various local bifurcations, and more generally the possibility of multiple steady states, oscillations or chaos. Given a set $\cal{X}$ and a map $J:{\cal X} \to \mathbb{R}^{n \times n}$, define
\[
{\cal J} = \{J(x)\,:\, x \in {\cal X}\}\, \quad \mbox{and} \quad \mathrm{Spec}\,{\cal J} = \bigcup_{x \in {\cal X}}\mathrm{Spec}\,J(x)\,.
\]
We may ask, for example: 
\begin{enumerate}
\item[Q1.] ``Is $0 \in \mathrm{Spec}\,{\cal J}$?''
\item[Q2.] ``Does $ \mathrm{Spec}\,{\cal J}\backslash\{0\}$ intersect the imaginary axis?''
\item[Q3.] ``Does the nonreal part of $ \mathrm{Spec}\,{\cal J}$ intersect both left and right open half-planes of $\mathbb{C}$?''
\end{enumerate}
and so forth. The best-known special case is where ${\cal J}$ is a ``qualitative class'', namely it consists of all matrices with some sign pattern, in which case Q1 reduces to the well-studied {\em combinatorial} problem of characterising sign nonsingular matrices \cite{brualdi}. By an easy argument, given convex open $Y \subseteq \mathbb{R}^n$, and a $C^1$ function $f:Y \to \mathbb{R}^n$, sign nonsingularity of the Jacobian matrices $Df$ in fact implies injectivity of $f$ \cite{gouze98}. Apart from such applications, exploration of the zero patterns of sign nonsingular matrices has led to a rich combinatorial theory \cite{RobertsonSeymourThomas}.

If ${\cal J}$ arises as the Jacobian matrices of some family of vector fields, a negative answer to either Q2 or Q3 implies that this family does not admit Hopf bifurcation. In this paper it will be Q3 which is of primary interest. 

The philosophical approach taken is quite analogous to that taken in the study of sign nonsingularity. As a general principle, when posing some question about a set of matrices $\mathcal{J}$, one hopes to associate with elements of $\mathcal{J}$ discrete objects such as graphs which are constant or vary little over $\mathcal{J}$, and then reduce the question to some finite computation on these objects. The most obvious example is again when $\mathcal{J}$ is a qualitative class, which can naturally be associated with a signed digraph, or a signed bipartite graph. However, more generally, products of matrices which have constant sign pattern (or some of which are constant) also arise naturally in applications. To see why it is worth going beyond qualitative classes, we introduce the notion of an ``interaction network'' as described in \cite{banajicraciun2}. 

\subsection{Interaction networks} 

Consider a system consisting of $n$ species with values (amounts, concentrations, populations, etc.) $x_1, \ldots, x_n \in X \subseteq \mathbb{R}$. Define $x = (x_1, \ldots, x_n)^t \in X^n$ to be the state of the system. Suppose that amongst these species there are $m$ interactions, each involving some nonempty subset of the species, and occurring at rates $v_i:X^n \to \mathbb{R}$ ($i = 1, \ldots, m$) dependent on the state $x$, but independent of time. Define the rate function $v:X^n \to \mathbb{R}^m$ by $v(x) = (v_1(x), \ldots, v_m(x))^t$. Finally, for $i = 1, \ldots, n$, let $f_{i}:\mathbb{R}^m \to \mathbb{R}$ describe the total rate of change of species $i$ as a function of the interaction rates. The evolution of the system is then given by:
\begin{equation}
\label{genprodform}
\dot x_i = f_{i}(v(x)), \qquad i = 1, \ldots, n.
\end{equation}
or more briefly, defining $f:\mathbb{R}^m \to \mathbb{R}^n$ by $f(y) = (f_1(y), \ldots, f_n(y))^t$,
\begin{equation}
\label{genprodform1}
\dot x = f(v(x)).
\end{equation}
Assume that $f$ and $v$ are $C^1$ functions, so that by the chain rule, 
\begin{equation}
\label{decomp}
D(f(v(x)) = Df(v(x))Dv(x).
\end{equation}
For our purposes here, the important point is that the right hand side of an interaction network (\ref{genprodform1}) is defined as a composition of functions, and hence its Jacobian matrices have a natural factorisation (\ref{decomp}). This factorisation was the starting point for quite general combinatorial approaches to questions of injectivity and multistationarity in interaction networks in \cite{banajicraciun2}, and the treatment here relies heavily on the ideas first presented there. 

\begin{remark}
(\ref{genprodform1}) is not restrictive since an arbitrary system of autonomous first order differential equations can be represented in this way by choosing, for example, one of $v$ or $f$ to be the identity function. Thus casting a system as an interaction network cannot be done uniquely, and can be seen as a formalism for studying certain questions about the system, rather than a categorisation of the system. However the usefulness of this construction is most apparent when $f$ and $v$ are defined by natural physical constraints as we shall see below. 
\end{remark}

\subsection{Chemical reaction networks}

O.D.E. models of systems of chemical reactions, termed ``chemical reaction networks'' or CRNs, naturally take the interaction network form: in fact (\ref{genprodform1}) reduces to 
\begin{equation}
\label{genprodformCRN}
\dot x = \Gamma v(x).
\end{equation}
where $\Gamma$ is now a constant $n \times m$ matrix, termed the ``stoichiometric matrix'' of the system, and $v(x)$ is the vector of reaction rates. As the state variables are chemical concentrations, the natural state space is the nonnegative orthant in $\mathbb{R}^n$, i.e. $x \in \mathbb{R}^n_{\geq 0}$. The Jacobian matrices become
\[
\Gamma Dv(x). 
\]
Moreover, under mild physical assumptions, the sign pattern of $Dv$ may be related to that of $\Gamma$. For example, if $x \in \mathrm{int}\,\mathbb{R}^n_{\geq 0}$, no chemical occurs on both sides of any reaction, and all chemical reactions are reversible, then we expect $Dv_{ji}$ to have the same sign as $-\Gamma_{ij}$. This will be discussed further below.

\subsection{A roadmap for the paper}

The results will be developed as follows. In Section~\ref{sec:prelim} a considerable volume of preliminary material is gathered for subsequent use. This includes definitions, notation and known results, along with a few new but relatively straightforward lemmas. Section~\ref{sec:C2decomp} constructs an object termed the ``DSR$^{[2]}$ graph'', which is the central tool required for everything to follow. This is followed, in Section~\ref{sec:results1} by a series of claims which can be made immediately from observation of the DSR$^{[2]}$ graph combined with the results in Section~\ref{sec:prelim}, and without the need for deeper theory. Section~\ref{sec:theory} contains technical results on the DSR$^{[2]}$ graph which allow, in Section~\ref{sec:final} some nontrivial extensions of the results of earlier sections. Section~\ref{sec:final} also describes limitations of the theory developed, and suggests the way forward.

Examples are interspersed throughout, and we return to examples to note how the analysis becomes more sophisticated or more rapid as the theoretical tools are developed. Similarly we sometimes prove results which later become corollaries of deeper results; although this adds to the length of the paper, it is intended to make the material more transparent. Remarks linking results here to related published work, and also to the study of CRNs, are interspersed throughout. The frequent reference to CRNs reflects the fact that although these systems are only one possible application for the theory, they were an important motivator for this work.

\begin{section}{Preliminaries}\label{sec:prelim}

\subsection{Some background on matrices}
This paper is concerned with sets of real matrices. These can be naturally identified with subsets of Euclidean space, thus inheriting topological properties such as openness, closedness and connectedness. We begin with the most fundamental of these sets: a matrix $M \in \mathbb{R}^{n \times m}$ determines the {\bf qualitative class} $\mathcal{Q}(M) \subseteq \mathbb{R}^{n \times m}$ consisting of all matrices with the same sign pattern as $M$, i.e., $X \in \mathcal{Q}(M)$ if and only if $(M_{ij} > 0) \Rightarrow (X_{ij} > 0)$; $(M_{ij} < 0) \Rightarrow (X_{ij} < 0)$; and $(M_{ij} = 0) \Rightarrow (X_{ij} = 0)$. The closure of $\mathcal{Q}(M)$ will be written $\mathcal{Q}_0(M)$, while the closure of a more general $\mathcal{M}\subseteq \mathbb{R}^{n \times m}$ will be written $\overline {\mathcal{M}}$. A square matrix $M$ is {\bf sign nonsingular} if all matrices in $\mathcal{Q}(M)$ are nonsingular and {\bf sign singular} if all matrices in $\mathcal{Q}(M)$ are singular \cite{brualdi}.

\begin{remark}
\label{rem:CRNJacobian}
{\bf Jacobian matrices of CRNs.} Returning to (\ref{genprodformCRN}) and the subsequent discussion, the notation developed above allows us to abbreviate the condition ``$Dv_{ji}$ has the same sign as $-\Gamma_{ij}$ for all $i,j$'' as $Dv \in \mathcal{Q}(-\Gamma^t)$. More generally, for systems of reactions where no chemical occurs on both sides of any reaction, we expect the condition
\begin{equation}
\label{eq:CRNJac}
Dv \in \mathcal{Q}_0(-\Gamma^t)\,
\end{equation}
to hold on the entire nonnegative orthant, including the boundary. Even when this does not hold (namely when some chemical occurs on both sides of some reaction), provided every reversible reaction is treated as a pair of irreversible reactions, we still expect $Dv$ to belong to the closure of a qualitative class which can be inferred from the network structure. Thus quite generally, the Jacobian matrices of a CRN take the form of a constant matrix times a matrix belonging to some qualitative class. 
\end{remark}

\subsubsection{Submatrices, minors and terms in a determinant.} Given an $n \times m$ matrix $M$, and nonempty $\alpha \subseteq \{1, \ldots, n\}$, $\beta \subseteq \{1, \ldots, m\}$, define $M(\alpha|\beta)$ to be the submatrix of $M$ obtained by choosing rows of $M$ from $\alpha$ and columns from $\beta$, and (provided $|\alpha| = |\beta|$) $M[\alpha|\beta]$ to be the corresponding minor of $M$. $M[\alpha]$ will be an abbreviation for the principal minor $M[\alpha|\alpha]$. Given an $n \times n$ matrix $M$ and some permutation $\sigma$ of $\{1, \ldots, n\}$ denote by $P(\sigma)$ the sign of this permutation, and by $M_{(\sigma)}$ the term ${P(\sigma)}\Pi_{i=1}^nM_{i\sigma_i}$ in $\mathrm{det}\,M$, so that $\mathrm{det}\,M = \sum_\sigma M_{(\sigma)}$.

\subsubsection{The Cauchy-Binet formula.} Consider an $n \times k$ matrix $A$ and an $k \times m$ matrix $B$ for arbitrary positive integers $n,m,k$. The Cauchy-Binet formula (\cite{gantmacher} for example) tells us that for any nonempty $\alpha \subseteq \{1, \ldots, n\}$, $\beta \subseteq \{1, \ldots, m\}$ with $|\alpha| = |\beta|$:
\begin{equation}
\label{eqCB0}
(AB)[\alpha|\beta] = \sum_{\substack{\gamma \subseteq \{1, \ldots, m\}\\ |\gamma| = |\alpha|}}A[\alpha|\gamma]B[\gamma|\beta]. 
\end{equation}
This formula for the minors of a product is central to many of the matrix and graph-theoretic results developed and cited below. This formula also has a briefer statement in terms of compound matrices mentioned later.

\subsubsection{Matrix spectra and stability.} Given a square matrix $M$, define $\mathrm{Spec}(M)$ to be the multiset of eigenvalues of $M$. Denote by $\mathbb{C}_{-}$ (resp. $\overline{\mathbb{C}_{-}}$) the open (resp. closed) left half-plane of $\mathbb{C}$ with $\mathbb{C}_{+}$ and $\overline{\mathbb{C}_{+}}$ similarly defined. A square matrix $M$ whose spectrum lies entirely in $\mathbb{C}_{+}$ (resp. $\overline{\mathbb{C}_{+}}$) is {\bf positive stable} (resp. {\bf positive semistable}). $M$ is {\bf Hurwitz} if all of its eigenvalues lie in $\mathbb{C}_{-}$, which occurs iff $-M$ is positive stable. {\bf $P$-matrices} are square matrices all of whose principal minors are positive. {\bf $P_0$-matrices} are matrices in the closure of the $P$-matrices, i.e., all of whose principal minors are nonnegative. Here a $P$-matrix (or $P_0$-matrix) will refer to a real matrix. Of particular importance will be the following restrictions on the spectra of $P_0$-matrices and $P$-matrices proved in \cite{kellogg}:
\begin{lemma1}
\label{P0eigenvals}
A complex number $\lambda = re^{i\theta}$ is an eigenvalue of an $n \times n$ $P_0$-matrix if and only if $r=0$ or
\begin{equation}
\label{eg:P0wedge}
|\theta - \pi| \geq \pi/n\,.
\end{equation}
$\lambda$ is an eigenvalue of an $n \times n$ $P$-matrix if and only if $r > 0$ and 
\[
|\theta - \pi| > \pi/n\,.
\]
\end{lemma1}
This implies, in particular that no $P_0$-matrix can have a negative real eigenvalue. Obviously, if $-J$ is a $P_0$-matrix, then $J$ cannot have a positive real eigenvalue. We will also need the following well-known fact about the spectrum of a product of matrices:
\begin{lemma1}
\label{spectranposelemma}
Given $A \in \mathbb{R}^{n \times m}, B \in \mathbb{R}^{m \times n}$, $\mathrm{Spec}(AB)\backslash\{0\} = \mathrm{Spec}(BA)\backslash\{0\}$.
\end{lemma1}

\subsubsection{Compound matrices.} We define $\Lambda^k\mathbb{R}^n$ as the $k$th exterior power of $\mathbb{R}^n$. In particular, $\Lambda^0\mathbb{R}^n = \mathbb{R}$, $\Lambda^1\mathbb{R}^n = \mathbb{R}^n$, and for $k \geq 2$, the elements of $\Lambda^k\mathbb{R}^n$ are finite formal linear combinations of elements of the form
\[
u_1 \wedge u_2\wedge\ldots \wedge u_k
\]
where $u_i \in \mathbb{R}^n$, and ``$\wedge$'' is the wedge-product (see \cite{darling} for example). Let $M$ denote a linear transformation from $\mathbb{R}^m$ to $\mathbb{R}^n$, and also its matrix representation in some basis. $M$ determines transformations from $\Lambda^k\mathbb{R}^m$ to $\Lambda^k\mathbb{R}^n$ in two important ways \cite{li_muldowney_2000}:
\begin{enumerate}
\item $M^{(k)}:\Lambda^k\mathbb{R}^m \to \Lambda^k\mathbb{R}^n$ is defined by:
\[
M^{(k)}(u_1 \wedge \cdots \wedge u_k) = (Mu_1) \wedge \cdots \wedge (Mu_k)\,.
\]
In this case the map $M^{(k)}$ is termed the $k$th exterior power of $M$, or the $k$th multiplicative compound of $M$. Choosing some bases for $\Lambda^k\mathbb{R}^m$ and $\Lambda^k\mathbb{R}^n$, the resulting ${n\choose k} \times {m\choose k}$ matrix will also be denoted $M^{(k)}$. 
\item If $m = n$, $M^{[k]}:\Lambda^k\mathbb{R}^n \to \Lambda^k\mathbb{R}^n$ is defined by:
\[
M^{[k]}(u_1 \wedge \cdots \wedge u_k) = \sum_{i=1}^ku_1 \wedge \cdots (Mu_i) \wedge \cdots \wedge u_k\,.
\]
In this case the map $M^{[k]}$ is termed the $k$th additive compound of $M$. Choosing a basis for $\Lambda^k\mathbb{R}^n$, the resulting matrix is ${n\choose k} \times {n\choose k}$, and will also be denoted $M^{[k]}$. Note that $(-M)^{[2]} = -(M^{[2]})$ so we can write without ambiguity $-M^{[2]}$. 
\end{enumerate}

\begin{remark}
Multiplicative and additive compound matrices can easily be shown to satisfy the relationships $(AB)^{(k)} = A^{(k)}B^{(k)}$ (provided $A$ and $B$ are of dimensions such that $AB$ makes sense) and $(A+B)^{[k]} = A^{[k]}+ B^{[k]}$ (provided $A$ and $B$ are square and of the same dimensions). $(AB)^{(k)} = A^{(k)}B^{(k)}$ is in fact a restatement of the Cauchy-Binet formula in terms of multiplicative compounds. 
\end{remark}

\begin{remark}
When referring to a compound matrix we assume without comment that the dimensions are such that the matrix exists. For example, if $M^{[2]}$ mentioned, it is to be assumed that $M$ is a square matrix of dimension at least $2$. 
\end{remark}

\subsubsection{Spectra of compound matrices.} Compound matrices of a square matrix $M \in \mathbb{R}^{n\times n}$ appear naturally in the study of various questions on the stability of differential equations (see \cite{muldowney} for example). We will refer to some of this theory, but most important here are basic spectral properties of compound matrices: for $k = 1, \ldots, n$, the ${n\choose k}$ eigenvalues of $M^{(k)}$ are precisely {\em products} of all sets of $k$ distinct elements of  $\mathrm{Spec}(M)$, while the ${n\choose k}$ eigenvalues of $M^{[k]}$ are the {\em sums} of all sets of $k$ distinct elements of $\mathrm{Spec}(M)$.
Here it is second additive compound matrices which will be of most interest: if 
$\mathrm{Spec}(M)=\{\lambda_1,\ldots, \lambda_n\}$ then 
$$\mathrm{Spec}(M^{[2]})=\{\lambda_i+\lambda_j:1\le i<j\le n\}.$$
Second additive compound matrices are sometimes presented by means of a {\em bialternate product} \cite{Guckenheimer1997aa}. This reference also demonstrates that second additive compounds are natural objects to consider when studying Hopf bifurcation. For us, the following two observations are important: 

\begin{lemma1}
\label{J2nonsingular}
If $M$ is a real square matrix and $M^{[2]}$ is nonsingular, then $\mathrm{Spec}\,M\backslash\{0\}$ does not intersect the imaginary axis.
\end{lemma1}
\begin{proof}
If $\mathrm{Spec}\,M\backslash\{0\}$ intersects the imaginary axis, then $M$ has a pair of eigenvalues $\pm i \omega$ where $\omega \neq 0$. These sum to zero and hence $M^{[2]}$ is singular. 
\end{proof}

\begin{lemma1}
\label{J2P0}
If $M$ is a real square matrix and $M^{[2]}$ is a $P_0$-matrix, then the nonreal part of $\mathrm{Spec}\,M$ (i.e. $\mathrm{Spec}\,M\backslash \mathbb{R}$)  does not intersect $\mathbb{C}_{-}$. If $-M^{[2]}$ is a $P_0$-matrix, then the nonreal part of $\mathrm{Spec}\,M$ does not intersect $\mathbb{C}_{+}$. 
\end{lemma1}
\begin{proof}
If the nonreal part of $\mathrm{Spec}\,M$ intersects $\mathbb{C}_{-}$, then $M$ has a pair of eigenvalues $a \pm i \omega$ where $a < 0$. These sum to $2a < 0$ and hence $M^{[2]}$ has a negative real eigenvalue, forbidden by Lemma~\ref{P0eigenvals}. The second statement follows immediately since the spectrum of $-M^{[2]}$ is simply the reflection in the imaginary axis of that of $M^{[2]}$. 
\end{proof}

\begin{remark}
\label{rem:adddiag}
If $M^{[2]}$ is a $P_0$-matrix, then in fact Lemma~\ref{J2P0} extends immediately as follows: given any nonnegative diagonal matrix $D$, the nonreal part of $\mathrm{Spec}\,(M+D)$ does not intersect $\mathbb{C}_{-}$. This follows because (i) $(M+D)^{[2]} = M^{[2]} + D^{[2]}$, (ii) $D^{[2]}$ is itself a nonnegative diagonal matrix, and (iii) the set of $P_0$-matrices is closed under the addition of nonnegative diagonal matrices. This fact is of importance in applications to CRNs, where diagonal terms represent arbitrary degradation or outflow of chemicals (see \cite{craciun,banajicraciun}). This will be commented on further below.
\end{remark}

We also have the following two lemmas about stability:
\begin{lemma1}
\label{lemposstable0}
Consider a square matrix $M$. If both $M$ and $M^{[2]}$ are $P_0$-matrices (resp. $P$-matrices), then $M$ is positive semistable (resp. positive stable).
\end{lemma1}
\begin{proof}
Since $M$ is a $P_0$-matrix (resp. $P$-matrix) by Lemma~\ref{P0eigenvals} the real part of $\mathrm{Spec}\,M$ does not intersect $\mathbb{C}_{-}$ (resp. $\overline{\mathbb{C}_{-}}$). Since $M^{[2]}$ is a $P_0$-matrix (resp. $P$-matrix), by Lemma~\ref{J2P0}, the nonreal part of $\mathrm{Spec}\,M$ does not intersect $\mathbb{C}_{-}$ (resp. $\overline{\mathbb{C}_{-}}$). Combining these two facts, $M$ is positive semistable (resp. positive stable).
\end{proof}

\begin{remark}
\label{rem:adddiaga}
Continuing from Remark~\ref{rem:adddiag}, if the conditions of Lemma~\ref{lemposstable0} hold, namely both $M$ and $M^{[2]}$ are $P_0$-matrices (resp. $P$-matrices), then given any nonnegative diagonal matrix $D$, $M + D$ is positive semistable (resp. positive stable). 
\end{remark}

\begin{lemma1}
\label{lemposstable}
Consider a set $\mathcal{M} \subseteq \mathbb{R}^{n \times n}$ which is path-connected. Suppose that for each $M \in \mathcal{M}$, both $M$ and $M^{[2]}$ are nonsingular and that $\mathcal{M}$ includes a positive stable (resp. Hurwitz) matrix. Then $\mathcal{M}$ consists entirely of positive stable (resp. Hurwitz) matrices. 
\end{lemma1}
\begin{proof}
Let $M' \in \mathcal{M}$ be a positive stable (resp. Hurwitz) matrix which exists by assumption. Suppose $\mathcal{M}$ includes a matrix $M''$ which fails to be positive stable (resp. Hurwitz). Since $\mathcal{M}$ is path-connected, and eigenvalues of a matrix depend continuously on the entries of a matrix, there must lie, on any path in $\mathcal{M}$ connecting $M'$ and $M''$, a matrix $M \in \mathcal{M}$ with some eigenvalue on the imaginary axis. However, since both $M$ and $M^{[2]}$ are nonsingular, by Lemma~\ref{J2nonsingular}, $M$ has no eigenvalues on the imaginary axis. 
\end{proof}

\begin{remark}
The stability criterion in Lemma~\ref{lemposstable} is closely related to that of Li and Wang in \cite{li_wang}, and indeed, it is possible to use the spectral properties of compound matrices to develop other stability criteria for sets of matrices which can be more useful in practice than attempting directly to check the Routh-Hurwitz conditions.
\end{remark}

\subsection{Hopf bifurcations} 

A one-parameter family of vector fields $f(x, \mu)$ generically undergoes a Poincar\'e-Andronov-Hopf bifurcation (from here on abbreviated to ``Hopf bifurcation'') if there exists some $(x_0, \mu_0)$ such that $f(x_0, \mu_0) = 0$, $Df(x_0, \mu_0)$ has a pair of imaginary eigenvalues and moreover these eigenvalues cross the imaginary axis as $\mu$ varies through $\mu_0$ (\cite{Wiggins} for example). Thus vector fields can be shown to {\em forbid} Hopf bifurcation either if the nonzero spectrum of each Jacobian matrix avoids the imaginary axis, or the nonreal spectrum of each Jacobian matrix avoids $\mathbb{C}_{-}$ or $\mathbb{C}_{+}$.

Let $\mathbb{R}$ and $\Im$ be the real and imaginary axes in $\mathbb{C}$. As indicated earlier, we are particularly interested in matrix products. In particular, when studying Hopf bifurcation, we are led to the following questions: 
\begin{enumerate}
\item Given $\mathcal{A} \subseteq \mathbb{R}^{n \times m}$ and $\mathcal{B}\subseteq \mathbb{R}^{m \times n}$, is $\mathrm{Spec}(AB)\backslash\{0\} \cap \Im = \varnothing \,\, \mbox{for all} \,\, A \in \mathcal{A}, B \in \mathcal{B}$?
\item Given $\mathcal{A} \subseteq \mathbb{R}^{n \times m}$ and $\mathcal{B}\subseteq \mathbb{R}^{m \times n}$, is $\mathrm{Spec}(AB) \subseteq \mathbb{R} \cup \overline{\mathbb{C}_{+}}\,\, \mbox{for all} \,\, A \in \mathcal{A}, B \in \mathcal{B}$?
\end{enumerate}
Since under certain assumptions, the Jacobian matrix of a CRN is of the form $-AB$ where $B \in \mathcal{Q}_0(A^t)$ (see Remark~\ref{rem:CRNJacobian}), when analysing the behaviour of CRNs we are led to the questions:
\begin{enumerate}
\item Given $A \in \mathbb{R}^{n \times m}$, is $\mathrm{Spec}(AB)\backslash\{0\} \cap \Im = \varnothing\,\, \mbox{for all} \,\, B \in \mathcal{Q}_0(A^t)$?
\item Given $A \in \mathbb{R}^{n \times m}$, is $\mathrm{Spec}(AB) \subseteq \mathbb{R} \cup \overline{\mathbb{C}_{+}}\,\, \mbox{for all} \,\, B \in \mathcal{Q}_0(A^t)$?
\end{enumerate}
In this paper we begin the process of developing tools which provide answers to these questions. 

\begin{remark}
\label{rem:chemHopf}
From Lemma~\ref{J2P0}, the reader may already guess that our main tools for ruling out Hopf bifurcations here will involve proving that $J^{[2]}$ or $-J^{[2]}$ is a $P_0$-matrix for each allowed Jacobian matrix of some system. Continuing the theme of Remark~\ref{rem:adddiag}, if Hopf bifurcation is ruled out for a chemical system with Jacobian matrices $\mathcal{J}$ because $-J^{[2]}$ is a $P_0$-matrix for all $J \in \mathcal{J}$, then it is in fact ruled out even if arbitrary degradation or outflow reactions are added. (Analogous statements follow replacing ``$-J^{[2]}$'' with ``$-J$'' and ``Hopf bifurcation'' with ``saddle-node bifurcation'', but this is well-studied \cite{banajicraciun2,banajiSIAM} and not the theme of this paper.)
\end{remark}

As indicated earlier, second additive compounds arise naturally in the study of Hopf bifurcations, a theme discussed thoroughly in \cite{Guckenheimer1997aa}. In that work, the focus lies on devising continuation methods to detect Hopf bifurcations, with $J^{[2]}$ occurring in algebraic systems which augment the equilibrium condition. On the other hand, our goal here is to present conditions which rely on combinatorial features of $J^{[2]}$ to forbid Hopf bifurcation. These conditions translate into graphical requirements involving {\bf DSR graphs}, which are discussed next.

\subsection{The DSR graph}
\label{sec:DSR}

A class of generalised graphs, sometimes termed SR graphs and DSR graphs, have become useful tools in the study of properties of interaction networks, such as multistationarity and stability \cite{craciun1,banajicraciun,banajicraciun2,angelileenheersontag,craciunPanteaSontag,minchevaroussel}.

DSR graphs for general interaction networks were constructed originally in \cite{banajicraciun2}. The definition presented here involves some minor modifications of that in \cite{banajicraciun2} for ease of presentation and maximum generality. Given $A \in \mathbb{R}^{n \times m}$ and $B \in \mathbb{R}^{m \times n}$, the DSR graph $G_{A, B}$ is defined as follows: we begin with a signed bipartite digraph on $n+m$ vertices, $S_1, \ldots, S_n$, termed S-vertices, and $R_1, \ldots, R_m$, termed R-vertices, with arc $R_jS_i$ if and only if $A_{ij}\not = 0$, and arc $S_iR_j$ if and only if $B_{ji} \not = 0$. The arc $R_jS_i$ is said to have R-to-S orientation and is given the sign of $A_{ij}$; the arc $S_iR_j$ is said to have S-to-R orientation and is given the sign of $B_{ji}$. In addition:
\begin{enumerate}
\item A pair of antiparallel edges of the same sign are treated as a single undirected edge, so that each edge of $G_{A, B}$ either has S-to-R orientation or R-to-S orientation or both, in which case we say that the edge is undirected. When discussing the degree of a vertex, an undirected edge incident into the vertex contributes exactly $1$ to the degree. 
\item It is useful to give the (directed or undirected) edge $R_jS_i$ associated with a nonzero entry $A_{ij}$ the label $|A_{ij}|$. We refer to this label of an edge $e$ as $l(e)$. If $e$ has only S-to-R orientation (i.e. it is associated only with an entry from $B$), then we formally define $l(e) = \infty$. This choice of labelling convention springs from the fact that we can generally consider $A$ to be constant if the system being studied is a CRN (see Remark~\ref{rem:CRNJacobian}). 
\end{enumerate}

\begin{remark}
\label{rem:DSRterminology}
What is here termed $G_{A, B}$ would be $G_{A, B^t}$ in the terminology of \cite{banajicraciun2}. Further, the construction here is presented for a single pair of matrices rather than directly for sets of matrices; note however that if $A$ is fixed and $B$ varies over a qualitative class, then the DSR graph remains constant. On the other hand, if $A$ varies (within a qualitative class), then varying edge labels can be replaced with the formal label $\infty$. 
\end{remark}

\begin{example}\label{figDSR}
Below we illustrate the DSR graph $G_{A,B}$ associated with two matrices $A$ and $B$. Negative edges (corresponding to negative entries in the matrices) are shown as dashed lines, while positive edges are bold lines, a convention which will be followed throughout this paper. To make the construction more transparent, an intermediate stage is shown where separate signed digraphs $G_{A, 0}$ and $G_{0, B}$ are constructed from $A$ and from $B$: $G_{A,B}$ can be regarded as the ``superposition'' of these two digraphs. Edge-labels are associated only with $A$: an edge imported only from $B$ is labelled $\infty$.

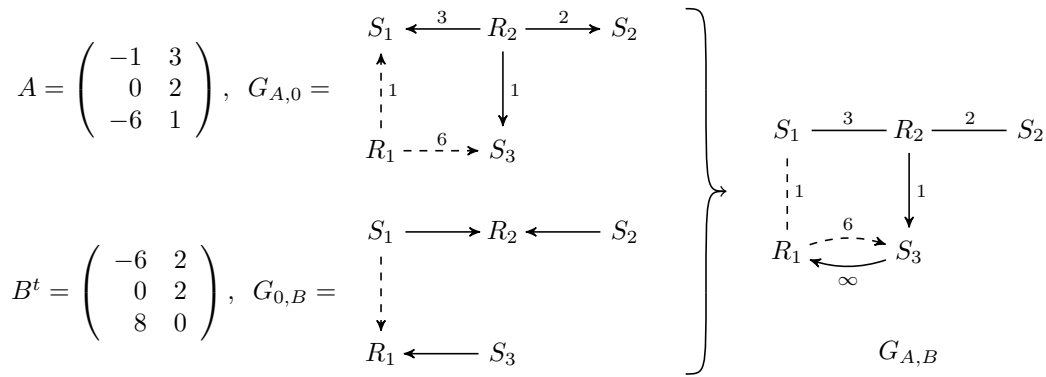
\begin{center}

\begin{tikzpicture}
[scale=0.54, place/.style={circle,draw=blue!50,fill=blue!5,thick,inner sep=0pt,minimum size=5.5mm},
place1/.style={circle,draw=white!5,fill=white!9,thick,inner sep=0pt,minimum size=5.5mm},
transition/.style={rectangle,draw=black!50,fill=black!5,thick,inner sep=0pt,minimum size=5mm},
transition1/.style={rectangle,draw=black!0,fill=black!0,thick,inner sep=0pt,minimum size=5mm},
pre/.style={->,shorten <=1pt,shorten >=1pt,>=stealth',semithick},
pres/.style={->,shorten <=-1pt,shorten >=-1pt,>=stealth',semithick},
pre1/.style={<-,shorten <=1pt,shorten >=1pt,>=stealth',semithick},
pres1/.style={<-,shorten <=-1pt,shorten >=0pt,>=stealth',semithick}];

\node at (-4.1,7.5) {$A = \left(\begin{array}{rr}-1 & 3\\0 & 2\\-6 & 1 \end{array}\right),\,\,\, G_{A,0} = $};

\node[place1] (S1) at (1,9) {$S_1$};
\node[transition1] (R2) at (4,9) {$R_2$};
\node[place1] (S2) at (7,9) {$S_2$};
\node[transition1] (R1) at (1,6) {$R_1$};
\node[place1] (S3) at (4,6) {$S_3$};

\draw[pre1,->] (R2) -- (S1);
\draw[pre1,->] (R2) -- (S2);
\draw[pre1,->] (R2) -- (S3);

\draw[pre1,dashed,->] (R1) -- (S1);
\draw[pre1,dashed,->] (R1) -- (S3);

\node at (2.5, 9.3) {$\scriptstyle{3}$};
\node at (2.5, 6.3) {$\scriptstyle{6}$};

\node at (1.3, 7.5) {$\scriptstyle{1}$};
\node at (4.3, 7.5) {$\scriptstyle{1}$};

\node at (5.5, 9.3) {$\scriptstyle{2}$};

\node at (-4.1,2.5) {$B^t = \left(\begin{array}{rr}-6 & 2\\0 & 2\\8 & 0 \end{array}\right),\,\,\, G_{0, B} = $};

\node[place1] (S1) at (1,4) {$S_1$};
\node[transition1] (R2) at (4,4) {$R_2$};
\node[place1] (S2) at (7,4) {$S_2$};
\node[transition1] (R1) at (1,1) {$R_1$};
\node[place1] (S3) at (4,1) {$S_3$};

\draw[pre1,<-] (R2) -- (S1);
\draw[pre1,<-] (R2) -- (S2);

\draw[pre1,dashed,<-] (R1) -- (S1);
\draw[pre1,<-] (R1) -- (S3);

\draw[-, semithick] (8.5, 0.5) .. controls (9,0.5) and (9, 0.5) .. (9,2.5);
\draw[->, semithick] (9, 2.5) .. controls (9,5) and (9, 5) .. (9.5,5);
\draw[-, semithick] (8.5, 9.5) .. controls (9,9.5) and (9, 9.5) .. (9,7.5);
\draw[-, semithick] (9, 7.5) .. controls (9,5) and (9, 5) .. (9.5,5);

\node[place1] (S1) at (11,6.5) {$S_1$};
\node[transition1] (R2) at (14,6.5) {$R_2$};
\node[place1] (S2) at (17,6.5) {$S_2$};
\node[transition1] (R1) at (11,3.5) {$R_1$};
\node[place1] (S3) at (14,3.5) {$S_3$};

\draw[pre1,-] (R2) -- (S1);
\draw[pre1,-] (R2) -- (S2);
\draw[pre1,->] (R2) -- (S3);

\draw[pre1,dashed,-] (R1) -- (S1);

\draw[pre1,->, dashed] (R1) .. controls (12.2,3.9) and (12.8, 3.9) .. (S3);
\draw[pre1,<-] (R1) .. controls (12.2,3.1) and (12.8, 3.1) .. (S3);

\node at (12.5, 6.8) {$\scriptstyle{3}$};
\node at (12.5, 4.2) {$\scriptstyle{6}$};
\node at (12.5, 2.8) {$\scriptstyle{\infty}$};

\node at (11.3, 5) {$\scriptstyle{1}$};
\node at (14.3, 5) {$\scriptstyle{1}$};

\node at (15.5, 6.8) {$\scriptstyle{2}$};

\node at (14, 1) {$G_{A,B}$};

\end{tikzpicture}
\end{center}
\end{example}

\begin{remark}
\label{remsignpat}
A special case relevant to the study of CRNs is when the matrices $A, B$ satisfy $B \in \mathcal{Q}(A^t)$. In this case all edges of $G_{A,B}$ are undirected, and in fact $G_{A,B} = G_{A, A^t}$. This graph can also be referred to as the {\bf SR graph} corresponding to the matrix $A$.  
\end{remark}

To present the main results about DSR graphs, some further definitions are needed. 

\subsubsection{Walks and cycles.} A {\bf walk} $W$ in a digraph is an alternating sequence of vertices and edges, beginning and ending with a vertex, and where each edge in $W$ is preceded by its start-point and followed by its end-point, each vertex (except the first) is preceded by an edge incident into it, and each vertex (except the last) is followed by an edge incident out of it. Here we allow a walk $W$ to repeat both vertices and edges, and its {\bf length} $|W|$ is the number of edges in $W$, counted with repetition. If the first and last vertex are the same, the walk is called {\bf closed}. In what follows we may refer to walks by their sequence of edges, or vertices, or both and we will say that a walk is empty and denote it as $\varnothing$ if it contains no edges (an empty walk may include a single vertex). When used without qualification, ``walk'' will mean a nonempty walk. The {\bf sign} of a walk is defined as the product of signs of the multiset of edges in the walk. The empty walk is formally given the sign $+1$.

A closed nonempty walk which does not repeat vertices (except in the trivial sense that the first and last vertices are the same), is called a {\bf cycle}. Clearly a cycle has no repeated edges. A DSR graph with no cycles will be termed acyclic: note that $(u,v,u)$ is considered a cycle in a DSR graph if and only if $uv$ and $vu$ are directed edges with different signs and so are not treated as a single edge.

\subsubsection{Parity of walks, e-cycles and o-cycles.} As DSR graphs are bipartite objects, each closed walk $W$ has even length. Any walk of even length can be given a parity $P(W)$ defined as:
\[
P(W) = (-1)^{|W|/2}\mathrm{sign}(W)\,.
\]
If $P(W) = 1$, then $W$ is termed {\bf even}, and if $P(W) = -1$, then $W$ is termed {\bf odd}. A cycle $C$ is termed an {\bf e-cycle} if $P(C) = 1$, and an {\bf o-cycle} if $P(C) = -1$.

\subsubsection{Closed s-walks and s-cycles.} A closed walk $W = (e_1, e_2, \ldots, e_{2r})$ is an {\bf s-walk} if each edge $e$ in $W$ has an associated label $l(e) \neq \infty$ and
\[
\prod_{i = 1}^{r}l(e_{2i-1}) = \prod_{i = 1}^{r}l(e_{2i}).
\]
If $W$ is a cycle which is an s-walk, it is termed an {\bf s-cycle}. 

\subsubsection{Intersections of cycles.} A cycle in a DSR graph which includes some edge with only S-to-R orientation or only R-to-S orientation has a unique orientation; otherwise the cycle has two possible orientations. Two oriented cycles in a DSR graph are compatibly oriented, if each induces the same orientation on every edge in their intersection. Two cycles (possibly unoriented) are compatibly oriented if there is an orientation for each so that this requirement is fulfilled. In a DSR graph, two cycles have {\bf odd intersection} if they are compatibly oriented and each component of their intersection has odd length. Note that odd intersection was termed ``S-to-R intersection'' in \cite{craciun1,banajicraciun2}.

\begin{remark}
\label{remoddintersect}
A necessary condition for the existence of two cycles, say $C$ and $D$, with odd intersection in a DSR graph is the existence of at least one S-vertex of degree $3$ or more, and of at least one R-vertex of degree $3$ or more. This follows since the vertices at the ends of a component of the intersection of $C$ and $D$ of odd length must have degree at least $3$, and one of these must be an S-vertex and one an R-vertex.
\end{remark}

\subsubsection{Matchings.} A matching in a DSR graph is a set of edges without common vertices. An S-to-R matching is a matching all of whose edges have S-to-R orientation, and an R-to-S matching is similarly defined. A matching which covers all vertices is perfect.

\subsection{Some useful results on matrix sets and DSR graphs}
\nopagebreak

We will refer to a DSR graph $G$ as {\bf odd} if $G$ has no e-cycles. We will refer to it as {\bf odd$^*$} if all e-cycles in $G$ are s-cycles, and no two e-cycles have odd intersection. It will be called {\bf steady} if all cycles in $G$ are s-cycles. We summarise several important results connecting DSR graphs and $P_0$-matrices in a single theorem. 
\begin{thm}
\label{mainP0thm}
Define the following conditions on matrices $A \in \mathbb{R}^{n \times m}$ and $B \in \mathbb{R}^{m \times n}$ and associated graphs:
\begin{itemize}
\item[{\bf C1.}] $G_{A, B}$ is odd.
\item[{\bf C2.}] $G_{A, B}$ is odd$^*$.
\item[{\bf C3.}] $A'B'$ is a $P_0$-matrix for all $A'\in \mathcal{Q}_0(A)$, $B' \in \mathcal{Q}_0(B)$.
\item[{\bf C4.}] $AB'$ is a $P_0$-matrix for all $B' \in \mathcal{Q}_0(B)$.
\item[{\bf C5.}] $G_{A, A^t}$ is odd.
\item[{\bf C6.}] $G_{A, A^t}$ is odd$^*$.
\item[{\bf C7.}] Every square submatrix of $A$ is either sign nonsingular or sign singular. 
\item[{\bf C8.}] Every square submatrix of $A$ is either sign nonsingular or singular.
\item[{\bf C9.}] $A'C$ is a $P_0$-matrix for all $A'\in \mathcal{Q}_0(A)$, $C \in \mathcal{Q}_0(A^t)$.
\item[{\bf C10.}] $AC$ is a $P_0$-matrix for all $C \in \mathcal{Q}_0(A^t)$.
\end{itemize}
The following implications hold:
\[
\begin{array}{ccc}
C1 & \Leftrightarrow & C3\\
\Downarrow & & \Downarrow\\
C2 & \Rightarrow & C4.
\end{array}
\]
\[
\begin{array}{ccccc}
C5 & \Leftrightarrow & C7 & \Leftrightarrow & C9\\
\Downarrow & & \Downarrow & & \Downarrow\\
C6 & \Rightarrow & C8 & \Leftrightarrow & C10.
\end{array}
\]
\end{thm}
\begin{proof}
The implications $C1 \Rightarrow C2$, $C3 \Rightarrow C4$, $C5 \Rightarrow C6$, $C7 \Rightarrow C8$ and $C9 \Rightarrow C10$ follow by definition, while the other implications are results in \cite{banajiSIAM,banajicraciun,banajicraciun2}, or are immediate consequences of these results. Note that $C5 \Leftrightarrow C9$ and $C6 \Rightarrow C10$ are just immediate corollaries of $C1 \Leftrightarrow C3$ and $C2\Rightarrow C4$ respectively.
\end{proof}

\begin{remark}
\label{rem:P0proof}
Although some of the proofs of results in Theorem~\ref{mainP0thm} are lengthy, the spirit is worth outlining. Consider matrices $A,B$ and DSR graph $G_{A, B}$. Since the results are about the product $AB$, they begin with the Cauchy-Binet formula (\ref{eqCB0}), which provides an expression for each minor of $AB$ as a sum of products of minors of $A$ and $B$. The transition from matrices to DSR graphs is via the association of terms in minors with matchings in DSR graphs. Given appropriate sets $\alpha, \beta$ a term in $A[\alpha|\beta]$ corresponds to an R-to-S matching in $G_{A, B}$ which is perfect on the vertices $\{S_i\}_{i \in \alpha} \cup \{R_j\}_{j \in \beta}$, while a term in $B[\beta|\alpha]$ corresponds to an S-to-R matching which is perfect on the same vertex set. The union of these matchings is a subgraph of $G_{A, B}$ consisting of (disjoint) cycles and isolated, undirected, edges. Conclusions about the signs and values of terms in the minors follow from examination of these subgraphs.
\end{remark}

Alongside Theorem~\ref{mainP0thm}, we will need a few further results specifically about steady DSR graphs, which were not considered in the references above, but become important in the discussion here. 

\begin{lemma1}\label{lem:swalks}
Let $G$ be a steady DSR graph. Then any closed walk on $G$ is an s-walk.  
\end{lemma1}
\begin{proof}
Let $W$ be a closed walk on $G$. Certainly each edge of $W$ which occurs in a cycle of $G$ has finite label, since all cycles are s-cycles. Further, each edge of $W$ not occurring in a cycle of $G$ must be undirected, and consequently have a finite label: traversing arc $uv$ and returning subsequently to $u$ otherwise implies the existence of a cycle which includes $uv$. Thus all edges in $W$ have finite labels. We show that  
$$\lambda(W)\stackrel{\text{\tiny def}}{=}\left(\prod_{e\in\{\text{S-to-R edges of }W\}}l(e)\right)\left(\prod_{e\in\{\text{R-to-S edges of }W\}}l(e)\right)^{-1}=1.$$
If no vertex is repeated in $W$, then $W$ is either a cycle, in which case $\lambda(W)=1$ by hypothesis, or a single edge traversed back and forward, in which case again $\lambda(W)=1$. If, on the other hand, $W$ contains repeated vertices, then the shortest closed subwalk of $W,$ denoted $U$, is either a cycle or a single repeated edge. Denote by $W\backslash U$ the closed subwalk obtained by removing $U$ from $W$, namely if the vertex-list of $W$ is $(w_1, w_2, \ldots, w_r=w_1)$ and that of $U$ is $(w_j, \ldots, w_k=w_j)$, then for $W\backslash U$ we have $(w_1, \ldots, w_j, w_{k+1}, \ldots, w_r=w_1)$. Clearly
$$\lambda(W)=\lambda(U)\lambda(W\backslash U)=\lambda(W\backslash U).$$
If the walk $W\backslash U$ does not repeat vertices, we are done. Otherwise, we may repeat this procedure until the remaining walk has no repeated vertices, and the lemma is proved.  
\end{proof}

Since an acyclic DSR graph vacuously fulfils the conditions of Lemma~\ref{lem:swalks}, an immediate corollary is:
\begin{cor}\label{cor:tree-s}
Let $G$ be an acyclic DSR graph. Any closed walk on $G$ is an s-walk. 
\end{cor}

\begin{lemma1}
\label{lem:scyclePmat}
Consider some $M' \in \mathbb{R}^{n \times m}$, and some $\mathcal{M} \subseteq \mathcal{Q}(M')$ such that for each $M \in \mathcal{M}$, $G_{M, M^t}$ is steady. Given any $A, B^t \in \overline{\mathcal{M}}$, $AB$ is a $P_0$-matrix. 
\end{lemma1}
\begin{proof}
If the premise of the theorem implies that $AB$ is a $P_0$-matrix for each $A, B^t \in \mathcal{M}$, then the result follows for $A, B^t \in \overline {\mathcal{M}}$ since the set of $P_0$-matrices is closed in $\mathbb{R}^{n \times n}$. So consider $A, B^t \in {\mathcal{M}}$. Given any nonempty $\alpha \subseteq \{1, \ldots, n\}$, by the Cauchy-Binet formula:
\begin{equation}
\label{eqCB1}
(AB)[\alpha] = \sum_{\substack{\gamma \subseteq \{1, \ldots, m\}\\ |\gamma| = |\alpha|}}A[\alpha|\gamma]B[\gamma|\alpha]. 
\end{equation}
Fix $\gamma$, and for brevity define $M = A(\alpha|\gamma)$ and $k = |\alpha|$. Suppose $M_{(\sigma)}$ and $M_{(\tau)}$ are two nonzero terms in $\mathrm{det}\,M = A[\alpha|\gamma]$ corresponding to permutations $\sigma \neq \tau$ of $\{1, \ldots, k\}$. Note that
\[
|M_{(\sigma)}| = \prod_{i=1}^k |M_{i\sigma_i}|, \quad |M_{(\tau)}| = \prod_{i=1}^k |M_{i\tau_i}|\,.
\]
$M_{(\sigma)}$ and $M_{(\tau)}$ each correspond to a matching in $G_{A, A^t}$, and the union of these two matchings defines a set of disjoint cycles corresponding to nontrivial cycles in the permutation $\tau\circ\sigma^{-1}$ and isolated edges corresponding to trivial cycles in $\tau\circ\sigma^{-1}$ (see Remark~\ref{rem:P0proof}). Since $G_{A, A^t}$ is steady, we get $|M_{(\sigma)}| = |M_{(\tau)}|$ (see the proof of Lemma~7 in \cite{banajicraciun} for example). Thus $A[\alpha|\gamma] = r(\alpha, \gamma)|M_{(\sigma)}|$ where $r(\alpha, \gamma)$ is an integer independent of the values of nonzero entries in $A$. Since $B^t$ belongs to the same qualitative class as $A$ and all cycles in $G_{B^t,B}$ are s-cycles, we get $B[\gamma|\alpha] = r(\alpha, \gamma)|N_{(\sigma)}|$ where $N = B(\gamma|\alpha)$. Thus 
\[
A[\alpha|\gamma]B[\gamma|\alpha] = [r(\alpha, \gamma)]^2\,|M_{(\sigma)}|\,\,|N_{(\sigma)}| \geq 0\,,
\]
and the result that $AB$ is a $P_0$-matrix now follows from (\ref{eqCB1}).
\end{proof}

\end{section}

\begin{section}{Compound matrices and the DSR$^{[2]}$ graph}\label{sec:C2decomp}

The main purpose of this section is as follows: given a square matrix product $AB$, we wish to be able to write $(AB)^{[2]} = \tilde A \tilde B$ where $\tilde A$ depends only on $A$, and $\tilde B$ depends only on $B$. This will allow us to construct an object, the DSR$^{[2]}$ graph of $AB$, whose structure encodes information about $(AB)^{[2]}$ allowing the use of Theorem~\ref{mainP0thm} and subsequent results. To motivate the construction we first present two examples:

\begin{example}
\label{example_motivate}
Consider the matrices
\begin{equation}
\label{eq:motivate}
A = \left(\begin{array}{rrr}1&1&0\\1&1&1\\0&1&1\end{array}\right), \quad B^t = \left(\begin{array}{rrr}a&b&0\\c&d&e\\0&f&g\end{array}\right), \quad J = A B = \left(\begin{array}{ccc}a+b&c+d&f\\a+b&c+d+e&f+g\\b&d+e&f+g\end{array}\right).
\end{equation}
Assume that $a,b,c,d,e,f,g > 0$. Here, as in all examples to follow, we let $e_i$ denote the $i$th standard basis vector in $\mathbb{R}^n$ and choose the natural basis for $\Lambda^2\mathbb{R}^3$ consisting of the wedge-products of pairs of vectors $e_i$, arranged in lexicographic order. So for $\Lambda^2\mathbb{R}^3$, using the basis $(e_1 \wedge e_2, e_1 \wedge e_3, e_2 \wedge e_3)$, we can compute:
\begin{equation}
\label{eq:J2_motivate}
J^{[2]} = \left(\begin{array}{ccc}a+b+c+d+e & f+g & -f\\d+e & a+b+f+g & c+d\\-b & a+b & c+d+e+f+g\end{array}\right).
\end{equation}
By computing its $7$ principal minors, we can check directly that $J^{[2]}$ is a $P_0$-matrix (although $J$ itself need not be). We will see later that the conclusion that $J^{[2]}$ is a $P_0$-matrix can be reached without symbolic computation, and this example is a special case of a much more general result. 
\end{example}

\begin{example}
\label{example_motivate0}
Consider the matrices $A, B$ defined by:
\[
A = \left(\begin{array}{rccc}a&b&0\\-c&d&0\\0&e&f\\0&0&g\end{array}\right)\,,\quad B^t = \left(\begin{array}{rccc}a'&b'&0\\-c'&d'&0\\0&e'&f'\\0&0&g'\end{array}\right)\,,
\]
where $a,b,c,d,e,f,g,a',b',c',d',e',f',g'$ are all positive real numbers. $AB$ is a $4 \times 4$ matrix, so $(AB)^{[2]}$ is a $6 \times 6$ matrix. (i) We can check that $AB$ is a $P_0$-matrix. (ii) With the help of a symbolic algebra package such as MAXIMA \cite{maxima}, we can compute the $63$ principal minors of $(AB)^{[2]}$, confirm that these are all nonnegative, and conclude that $(AB)^{[2]}$ is a $P_0$-matrix. By Lemma~\ref{lemposstable0}, (i) and (ii) imply that $AB$ is positive semistable. It is again natural to ask whether this conclusion can be reached simply, via Theorem~\ref{mainP0thm}, and without the use of symbolic algebra. We will return to this example and show that this is indeed the case. 

\end{example}

\subsection{Identifying $\Lambda^2\mathbb{R}^n$ with antisymmetric matrices}

We will eventually write:
\[
(AB)^{[2]} = \overline{\mathbf{L}}^A\underline{\mathbf{L}}^B
\]
where the matrix $\overline{\mathbf{L}}^A$ depends on $A$ and $\underline{\mathbf{L}}^B$ depends on $B$. The linear algebraic justification for the factorisation begins by identifying $\Lambda^2\mathbb{R}^n$ with the set of $n \times n$ antisymmetric matrices. 

Define the matrix $X_{ij}$ by
\[
X_{ij} = e_ie_j^t-e_je_i^t\,.
\]
Note that $(X_{ij})^t = X_{ji} = -X_{ij}$. Denote by $\mathrm{asym}(n)$ the subspace of $\mathbb{R}^{n \times n}$ spanned by $\{X_{ij}\}$, i.e. the space of all real $n \times n$ antisymmetric matrices. Given any $C \in \mathbb{R}^{n \times n}$ define $L^{C}:\mathrm{asym}(n) \to \mathrm{asym}(n)$ via $L^C(X) = CX + XC^t$. Identifying $X_{ij} \in \mathrm{asym}(n)$ with $e_i \wedge e_j \in \Lambda^2\mathbb{R}^n$, $L^{C}$ can be identified with the second additive compound $C^{[2]}$. To be precise, define $\iota:\Lambda^2\mathbb{R}^n \to \mathrm{asym}(n)$ via $\iota(e_i \wedge e_j) = X_{ij}$. That this action on basis vectors extends naturally to give a well-defined linear bijection can easily be checked. We can then confirm that given any square matrix $C$, the following diagram commutes:\begin{center}
\begin{tikzpicture}[domain=0:4,scale=0.7]

\node at (1,4) {$\Lambda^2\mathbb{R}^n$};
\node at (4.2,4) {$\Lambda^2\mathbb{R}^n$};
\node at (0.8,1) {$\mathrm{asym}(n)$};
\node at (4.2,1) {$\mathrm{asym}(n)$};

\draw[->, semithick] (1.7,4.0) -- (3.5, 4.0);
\draw[->, semithick] (1.8,1.0) -- (3.2, 1.0);

\draw[->, semithick] (1.0, 3.6) -- (1.0, 1.4);
\draw[->, semithick] (4.0, 3.6) -- (4.0, 1.4);

\node at (0.6, 2.5) {$\iota$};
\node at (2.5, 0.6) {$L^C$};
\node at (2.5, 4.4) {$C^{[2]}$};
\node at (4.4, 2.5) {$\iota$};

\end{tikzpicture}
\end{center}
The calculation is straightforward:
\begin{eqnarray*}
(\iota \circ C^{[2]})(e_i\wedge e_j) & = &  \iota(Ce_i \wedge e_j + e_i \wedge Ce_j)\\
& = &\iota\left(\sum_kC_{ki}e_k \wedge e_j\right) + \iota\left(e_i \wedge \sum_kC_{kj}e_k\right)\\
& = &\sum_kC_{ki}\iota(e_k \wedge e_j) + \sum_kC_{kj}\iota(e_i \wedge e_k)\\
& = &\sum_kC_{ki}X_{kj} + \sum_kC_{kj}X_{ik}\\
& = &\sum_k\left[C_{ki}(e_ke_j^t-e_je_k^t) + C_{kj}(e_ie_k^t-e_ke_i^t)\right]\\
& = &Ce_ie_j^t- e_je_i^tC^t + e_ie_j^tC^t-Ce_je_i^t\\
& = & CX_{ij} + X_{ij}C^t =  L^CX_{ij} = (L^C\circ \iota)(e_i\wedge e_j)\,.
\end{eqnarray*}

\subsection{Factorising the second additive compound}

The construction is first carried out for rank $1$ matrices and extended to arbitrary matrices. 

Given $s \in \mathbb{R}^n$, we can define the linear map $\overline{\mathbf{L}}^s:\mathbb{R}^n \to \mathrm{asym}(n)$ via its action on coordinate vectors:
\[
\overline{\mathbf{L}}^se_k =  e_ks^t - se_k^t = \sum_le_ks_le_l^t - \sum_ls_le_le_k^t =  \sum_ls_lX_{kl}  = -\!\sum_{l=1}^{k-1}s_lX_{lk} + \sum_{l=k+1}^ns_lX_{kl}\,.
\]

Noting that if $k < l$ and $i<j$, then $e_i^tX_{kl}e_j=\delta_{ik}\delta_{jl}$ (where $\delta_{ik}$ is the Kronecker delta), a matrix representation of $\overline{\mathbf{L}}^s$ is obtained:
\begin{eqnarray}
\overline{\mathbf{L}}^s_{(i, j), k} & = & \left[\overline{\mathbf{L}}^se_k\right]_{(i, j)} \quad (i < j) \nonumber\\
& = & \sum_{l=k+1}^ns_le_i^tX_{kl}e_j - \sum_{l=1}^{k-1}s_le_i^tX_{lk}e_j \nonumber\\
\label{eq:matrep}
& = & \sum_{l=k+1}^ns_l\delta_{ki}\delta_{jl} - \sum_{l=1}^{k-1}s_l\delta_{kj}\delta_{il}  = \left\{\begin{array}{cc}s_j & (k= i)\\-s_i&(k=j)\\0&\mbox{otherwise}\end{array}\right.
\end{eqnarray}
For example if $s = (s_1,s_2,s_3,s_4)^t$, then
\[
\overline{\mathbf{L}}^s = \bordermatrix{{}&1&2&3 &4\cr
                (1,2)&s_2&-s_1&0&0\cr
                (1,3)&s_3&0&-s_1&0\cr
                (1,4)&s_4&0&0&-s_1\cr
                (2,3)&0&s_3&-s_2&0\cr
                (2,4)&0&s_4&0&-s_2\cr
                (3,4)&0&0&s_4&-s_3}
\]
Row and column labels have been shown explicitly to highlight the fact that each column is associated with an entry in $s$, while each row is associated with a {\em pair} of entries. If $v$ is a $1\times n$ matrix (a row vector) then  $\underline{\mathbf{L}}^v:\mathrm{asym}(n)\to \mathbb R^n$ is similarly defined via its action on basis elements of $\mathrm{asym}(n)$:
\[
\underline{\mathbf{L}}^vX_{ij} = X_{ij}v^t = (e_ie_j^t - e_je_i^t)v^t =  e_iv_j-e_jv_i\,.
\]
With $i < j$, we obtain the coordinate representation:
\[
\underline{\mathbf{L}}^v_{k, (i, j)} = \left[\underline{\mathbf{L}}^vX_{ij}\right]_k = e_k^t(e_iv_j - e_jv_i) = \delta_{ki}v_j - \delta_{kj}v_i = \left\{\begin{array}{cc}v_j & (k= i),\\-v_i&(k=j),\\0&\mbox{otherwise.}\end{array}\right.
\]
Thus $\underline{\mathbf{L}}^v = (\overline{\mathbf{L}}^{v^t})^t$.

Now consider some rank $1$ matrix $C = sv$ where $s, v^t \in \mathbb{R}^{n}$, and the associated operator $L^C:\mathrm{asym}(n) \to \mathrm{asym}(n)$. Applying the definitions:
\[
L^CX = svX + Xv^ts^t = (Xv^t)s^t -s(Xv^t)^t = (\underline{\mathbf{L}}^vX)s^t -s(\underline{\mathbf{L}}^{v}X)^t = \overline{\mathbf{L}}^s\underline{\mathbf{L}}^vX\,.
\]
In other words, $L^C = \overline{\mathbf{L}}^s\underline{\mathbf{L}}^v$. Thus representations of $\overline{\mathbf{L}}^s$ and $\underline{\mathbf{L}}^v$ above give us a natural representation and factorisation of $L^{sv}$. 

More generally, consider any $A \in \mathbb{R}^{n \times m}$ and $B \in \mathbb{R}^{m \times n}$. Defining $A_k$ to be the $k$th column of $A$ and $B^k$ to be the $k$th row of $B$ we have $AB = \sum_k A_kB^k$. We get
\begin{eqnarray*}
L^{AB}X & = & \sum_k(A_kB^kX + X(B^k)^t(A_k)^t)\\
& = & \sum_k((X(B^t)_k)(A_k)^t - A_k(X(B^t)_k)^t )\\
& = & \sum_k((\underline{\mathbf{L}}^{B^k}X)(A_k)^t -A_k(\underline{\mathbf{L}}^{B^k}X)^t ) = \sum_k\overline{\mathbf{L}}^{A_k}\underline{\mathbf{L}}^{B^k}X\,.
\end{eqnarray*}
Choosing a basis for $\Lambda^2\mathbb{R}^n$ we define for any $n \times m$ matrix $A$ the matrices
\[
\overline{\mathbf{L}}^A = \left[\overline{\mathbf{L}}^{A_1}|\overline{\mathbf{L}}^{A_2}|\cdots|\overline{\mathbf{L}}^{A_m}\right] \quad \mbox{and} \quad \underline{\mathbf{L}}^A = \left[\begin{array}{c}\underline{\mathbf{L}}^{A^1}\\\underline{\mathbf{L}}^{A^2}\\\vdots\\\underline{\mathbf{L}}^{A^n}\end{array}\right]\,.
\]
Note that $\underline{\mathbf{L}}^A = \left(\overline{\mathbf{L}}^{A^t}\right)^t$. Given $A\in\mathbb R^{n\times m}$ and $B\in\mathbb R^{m\times n}$ we get:
\begin{equation}
\label{eq:J2decomp}
(AB)^{[2]} = \overline{\mathbf{L}}^A\underline{\mathbf{L}}^B\,.
\end{equation}
Observe that $\overline{\mathbf{L}}^A$ has dimensions ${n \choose 2} \times mn$ and $\underline{\mathbf{L}}^B$ has dimensions $mn \times {n \choose 2}$.

\begin{remark}
\label{rem:CB}
When $A$ and $B$ are matrices with symbolic entries, one approach to computation of any minor of $(AB)^{[2]}$ is by application of the Cauchy-Binet formula to (\ref{eq:J2decomp}). For example assuming that $A$ is $n \times m$ and $B$ is $m \times n$, and defining $\alpha = \{1, \ldots, {n\choose 2}\}$:
\begin{equation}
\label{eqCB2}
\mathrm{det}\,((AB)^{[2]}) = \sum_{\substack{\beta \subseteq \{1, \ldots, nm\}\\ |\beta| = {n\choose 2}}}\overline{\mathbf{L}}^A[\alpha|\beta]\underline{\mathbf{L}}^B[\beta|\alpha]\,.
\end{equation}
This observation can be used in conjunction with Lemma~\ref{J2nonsingular} to rule out the possibility of pairs of imaginary eigenvalues of $AB$.
\end{remark}

\subsection{The DSR$^{[2]}$ graph}

Given $A \in \mathbb{R}^{n \times m}$ and $B \in \mathbb{R}^{m \times n}$ with $n \geq 2$ we define the {\bf DSR$^{[2]}$ graph} of the product $AB$
\[
G^{[2]}_{A, B} = G_{\overline{\mathbf{L}}^A,\, \underline{\mathbf{L}}^{B}}
\]
to be the DSR graph of the product $\overline{\mathbf{L}}^A\underline{\mathbf{L}}^{B}$. Examination of this graph allows us to make statements about $(AB)^{[2]}$ using, for example, Theorem~\ref{mainP0thm}.

\subsubsection{Terminology and notation.} The S-vertices of $G^{[2]}_{A, B}$ are the $n\choose 2$ unordered pairs $(i,j)\in\{1,\ldots,n\}\times\{1,\ldots,n\}$, $i\neq j$; the R-vertices of $G^{[2]}_{A, B}$ are the $mn$ pairs $(k,l)\in\{1,\ldots,m\}\times\{1,\ldots,n\}$. For notational convenience we denote S-vertices $(i,j)=(j,i)$ by $ij$, or $ji$;  R-vertices $(k,l)$ by $k^l$; and an edge between them of unknown direction and sign by $(ij,k^l).$ From the definitions, a necessary condition for the edge $(ij, k^l)$ to exist is $l \in \{i,j\}$. 

\begin{remark}
\label{remsignpat1}
Continuing from Remark~\ref{remsignpat}, it is clear that given a matrix $A$ and some $B \in \mathcal{Q}(A^t)$, $G^{[2]}_{A, B} = G^{[2]}_{A, A^t}$. It is also not hard to see that if $G^{[2]}_{A, A^t}$ is odd$^*$, then so is $G^{[2]}_{A, B}$ where $B \in \mathcal{Q}_0(A^t)$ ($G^{[2]}_{A, B}$ is just a subgraph of $G^{[2]}_{A, A^t}$). On the other hand, if $A$ has more than one row, and there exist $i,j$ such that $A_{ij}B_{ji} < 0$, then $G^{[2]}_{A, B}$ automatically fails Condition~C2. This follows because given any $k \ne i$ there exist two edges $(ik, j^k)$ in $G^{[2]}_{A, B}$, oppositely directed, of opposite sign, and one having label $\infty$; so $G^{[2]}_{A, B}$ has an e-cycle of length two which is not an s-cycle. In this paper, most claims will involve  matrix pairs $A,B$ satisfying $A_{ij}B_{ji} \geq 0$ (namely, $A,B^t \in \mathcal{Q}_0(C)$ for some matrix $C$), and often we will in fact assume $B \in \mathcal{Q}_0(A^t)$. 
\end{remark}
\end{section}

\begin{example}
\label{example_rank1}
If $A = (1,1,1,1)^t$, we can easily compute $G^{[2]}_{A, A^t}$ to be:
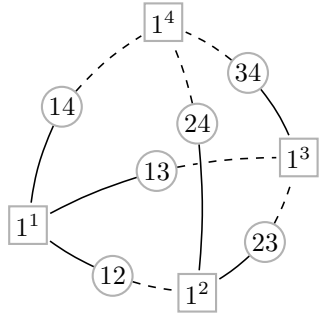
\begin{center}
\begin{tikzpicture}
[scale=.9, place/.style={circle,draw=black!30,fill=black!0 ,thick,inner sep=0pt,minimum size=5.3mm},
transition/.style={rectangle,draw=black!30,fill=black!0,thick,inner sep=0pt,minimum size=5mm},
pre/.style={<,shorten <=2pt,>=stealth',semithick},
post/.style={shorten >=1pt,>=stealth',semithick}]
\node[transition] (1T1)  at (0,1)  {$1^1$};
\node[transition] (1T2)  at (2.5,0)  {$1^2$};
\node[transition] (1T3)  at (4,2)  {$1^3$};
\node[transition] (1T4)  at (2,4)  {$1^4$};
\node[place] (12) at (1.25, 0.25) {$12$}
edge [bend left = 5, post]  (1T1)
edge [bend right = 5, dashed, post]  (1T2);
\node[place] (12) at (1.9, 1.8) {$13$}
edge [bend right = 5, post]  (1T1)
edge [bend left = 5, dashed, post]  (1T3);
\node[place] (12) at (0.5, 2.75) {$14$}
edge [bend right = 8, post]  (1T1)
edge [bend left = 8, dashed, post]  (1T4);
\node[place] (12) at (3.5, 0.75) {$23$}
edge [bend left = 5, post]  (1T2)
edge [bend right = 5, dashed, post]  (1T3);
\node[place] (12) at (2.5, 2.5) {$24$}
edge [bend left = 5, post, line width=0.08cm, color=white]  (1T2)
edge [bend left = 5, post]  (1T2)
edge [bend right = 5, dashed, post]  (1T4);
\node[place] (12) at (3.25, 3.25) {$34$}
edge [bend left = 7, post]  (1T3)
edge [bend right = 7, dashed, post]  (1T4);
\end{tikzpicture}

\end{center}
All edge-labels are $1$, so have been omitted. More generally it is clear that if $v \in \mathbb{R}^n$ has no zero entries, then (disregarding edge-signs and labels) $G^{[2]}_{v, v^t}$ is a subdivision of the complete graph $K_n$ on the R-vertices $1^1, \ldots, 1^n$. 

\end{example}

\begin{section}{Results illustrating the use of the DSR$^{[2]}$ graph}
\label{sec:results1}

Here we develop and apply some results using DSR$^{[2]}$ graphs before the development of further theory. Some of the examples presented will later be shown to fall into families amenable to analysis, after the development of theory relating DSR and DSR$^{[2]}$ graphs. The following rather basic result about the second additive compounds of certain rank $1$ matrices can easily be proved directly, but demonstrates application of the DSR$^{[2]}$ graph. A far-reaching generalisation of this result will be presented as Theorem~\ref{thm:tree}.
\begin{lemma1}
\label{lem:columnvec}
If $u \in \mathbb{R}^n$ and $s,v \in \mathcal{Q}_0(u)$, then $(sv^t)^{[2]}$ is a $P_0$-matrix.
\end{lemma1}
\begin{proof}
If the result is true for $s,v \in \mathcal{Q}(u)$, then it follows for $s,v \in \mathcal{Q}_0(u)$ since the set of $P_0$-matrices is closed. So consider $s,v \in \mathcal{Q}(u)$. By Remark~\ref{remsignpat1}, $G^{[2]}_{s, v^t} = G^{[2]}_{s,s^t}$. Note that from the definition of $\overline{\mathbf{L}}^s$, row $(j,k)$ ($j < k$) has at most two nonzero entries, the entry $s_k$ in column $j$ and the entry $-s_j$ in column $k$. Thus the edges incident with S-vertex $jk$ have labels $|s_j|$ and $|s_k|$. 
\begin{enumerate}
\item[(i)] No two cycles have odd intersection. Since no row of $\overline{\mathbf{L}}^s$ has more than two nonzero entries, each S-vertex in $G^{[2]}_{s,s^t}$ has maximum degree $2$, as illustrated for example by the DSR$^{[2]}$ graph of Example~\ref{example_rank1}. This rules out odd intersections of cycles (see Remark~\ref{remoddintersect}).
\item[(ii)] $G^{[2]}_{s,s^t}$ is steady. Define $\mathcal{I} = \{1, \ldots, n\}$, $\mathcal{I}_2 = \{(1,2), (1,3), \ldots, (n-1,n)\}$ (the set of all unordered pairs from $\mathcal{I}$). Let $i_1, \ldots, i_m \in \mathcal{I}$ be distinct, and $\alpha_1, \ldots, \alpha_m \in \mathcal{I}_2$ be distinct. An arbitrary cycle in $G^{[2]}_{s,s^t}$ must be of the form 
\[
C = (1^{i_1},\alpha_1,1^{i_2},\alpha_2,\cdots, 1^{i_m},\alpha_m,1^{i_1})
\]
where $i_k, i_{k+1} \in \alpha_k$, and since $i_k,i_{k+1}$ are distinct this implies that $\alpha_k = (i_k, i_{k+1})$. (Here we define $i_{m+1} = i_1$.) So the sequence of edge-labels in $C$ is
\[
(|s_{i_2}|, |s_{i_1}|, |s_{i_3}|, |s_{i_2}|, \cdots, |s_{i_{n}}|, |s_{i_{n-1}}|, |s_{i_1}|, |s_{i_n}|),
\]
and $C$ is clearly an $s$-cycle. 
\end{enumerate}
(i) and (ii) together imply that $G^{[2]}_{s,s^t}$ is odd$^*$, and hence by Theorem~\ref{mainP0thm} that $(sv^t)^{[2]}$ is a $P_0$-matrix. 
\end{proof}

\begin{example}[Example~\ref{example_motivate} continued]
\label{ex3row}
We now return to the matrices in (\ref{eq:motivate}). We will see that the DSR$^{[2]}$ graph makes our earlier conclusions immediate, but shows moreover that the example provides just a special case of a general result, Proposition~\ref{3sp1} below. Using (\ref{eq:matrep}), we compute
\[
\overline{\mathbf{L}}^A = \left(\begin{array}{rrrrrrrrr}1 & -1 & 0 & 1 & -1 & 0 & 1 & 0 & 0\\0 & 0 & -1 & 1 & 0 & -1 & 1 & 0 & 0\\0 & 0 & -1 & 0 & 1 & -1 & 0 & 1 & -1\end{array}\right)\,.
\]

The DSR graphs $G_{A, A^t}$ (left) and  $G^{[2]}_{A, A^t}$ (right) are depicted below:
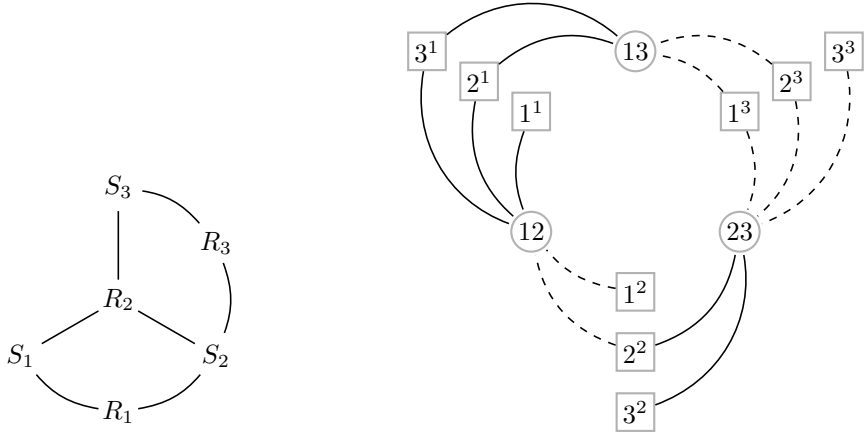
\begin{center}
\begin{tikzpicture}
[scale=0.75, place/.style={circle,draw=blue!50,fill=blue!5 ,thick,inner sep=0pt,minimum size=5mm},
place1/.style={circle,draw=white!5,fill=white!9,thick,inner sep=0pt,minimum size=5.5mm},
transition1/.style={rectangle,draw=black!0,fill=black!0,thick,inner sep=0pt,minimum size=5mm},
transition/.style={rectangle,draw=black!50,fill=black!5,thick,inner sep=0pt,minimum size=5.3mm},
pre/.style={<,shorten <=2pt,>=stealth',semithick},
post/.style={shorten >=1pt,>=stealth',semithick}]
\node[place1] (12) at (-1.732,-1) {$S_1$};
\node[place1] (23) at (1.732,-1) {$S_2$};
\node[place1] (13) at (0, 2) {$S_3$};
\node[transition1] ([13])  at (1.732,1)  {$R_3$}
edge [bend right = 20, post]  (13)
edge [bend left = 20, post]  (23);
\node[transition1] ([12])  at (0,- 2) {$R_1$}
edge [bend left = 20, post]  (12)
edge [bend right = 20, post] (23);
\node[transition1] ([11])  at (0,0) {$R_2$}
edge [post] (12)
edge [post] (13)
edge [post] (23);
\end{tikzpicture}
\hspace{2cm}
\begin{tikzpicture}
[scale=0.8, place/.style={circle,draw=black!30,fill=black!0 ,thick,inner sep=0pt,minimum size=5.3mm},
transition/.style={rectangle,draw=black!30,fill=black!0,thick,inner sep=0pt,minimum size=5mm},
pre/.style={<,shorten <=2pt,>=stealth',semithick},
post/.style={shorten >=1pt,>=stealth',semithick}]
\node[place] (12) at (-1.732,-1) {$12$};
\node[place] (23) at (1.732,-1) {$23$};
\node[place] (13) at (0, 2) {$13$};
\node[transition] ([13])  at (1.732,1)  {$1^3$}
edge [dashed,bend right = 20, post]  (13)
edge [dashed,bend left = 20, post]  (23);
\node[transition] ([12])  at (0,- 2) {$1^2$}
edge [dashed,bend left = 20, post]  (12);
\node[transition] ([11])  at (-1.732,1) {$1^1$}
edge [bend right = 20, post]  (12);
\node[transition] ([23])  at (1.5*1.732,1.5*1)  {$2^3$}
edge [dashed,bend right = 30, post]  (13)
edge [dashed,bend left = 30, post]  (23);
\node[transition] ([22])  at (1.5*0,- 1.5*2) {$2^2$}
edge [dashed,bend left = 30, post]  (12)
edge [bend right = 30, post]  (23);
\node[transition] ([21])  at (-1.5*1.732,1.5*1) {$2^1$}
edge [bend right = 30, post]  (12)
edge [bend left = 30, post]  (13);

\node[transition] ([33])  at (2*1.732,2*1)  {$3^3$}
edge [dashed,bend left = 40, post] (23);
\node[transition] ([32])  at (2*0,- 2*2) {$3^2$}
edge [bend right = 40, post] (23);
\node[transition] ([31])  at (-2*1.732,2*1) {$3^1$}
edge [bend right = 40, post] (12)
edge [bend left = 40, post] (13);

\end{tikzpicture}

\end{center}

All edge-labels are $1$ and have been omitted. Consequently, $G_{A, A^t}$ and $G^{[2]}_{A, A^t}$ are steady. $G_{A, A^t}$ includes the e-cycles $(S_1, R_1, S_2, R_2, S_1)$ and $(S_2, R_2, S_3, R_3, S_2)$ which have odd intersection and so fails to be odd$^*$. It is also easy to confirm that for some choice of $B \in \mathcal{Q}(A^t)$, $AB$ fails to be a $P_0$-matrix. On the other hand, since all R-vertices in $G^{[2]}_{A,A^t}$ have degree less than or equal to $2$, $G^{[2]}_{A, A^t}$ cannot have a pair of cycles with odd intersection, and so is odd$^*$ (see Remark~\ref{remoddintersect}). By Theorem~\ref{mainP0thm}, $(AB)^{[2]}$ is a $P_0$-matrix for all $B \in \mathcal{Q}_0(A^t)$, and so, by Lemma~\ref{J2P0}, the nonreal eigenvalues of $AB$ lie in $\overline{\mathbb{C}_{+}}$. 

\begin{remark}
\label{rem:global_stability}
In fact, if $B \in \mathcal{Q}(A^t)$, and the matrices $J=AB$ arise as the Jacobian matrices of a differential system, the form of $J^{[2]}$ in (\ref{eq:J2_motivate}) has strong consequences. In particular, suppose $Y \subseteq \mathbb{R}^3$ is open, $f:Y\to\mathbb{R}^3$ is $C^1$, and consider the system $\dot x = f(x)$ on $Y$. Let $X\subseteq Y$ be some compact, forward invariant, set, and assume that for each $x \in X$, $-Df(x)$ is of the form in (\ref{eq:motivate}). Then, by observation of (\ref{eq:J2_motivate}), $Df(x)^{[2]}$ is irreducibly diagonally dominant (in its columns) with negative diagonal entries, and it is possible to find a (single) positive diagonal matrix $E$ such that $E\,Df(x)^{[2]}\,E^{-1}$ is strictly diagonally dominant for all $x \in X$. Thus there exists a ``logarithmic norm'' $\mu_E(\,\cdot\,)$ and $p < 0$ such that for all $x \in X$, $\mu_E(Df(x)^{[2]}) < p$ \cite{strom}, and applying the theory developed by Li and Muldowney (\cite{li_muldowney_1993,li_muldowney_1996} for example), the system $\dot x = f(x)$ cannot have invariant closed curves in $X$: periodic orbits whether stable or unstable are ruled out, as are homoclinic trajectories, heteroclinic cycles, etc. In fact every nonwandering point of the system must be an equilibrium \cite{li_muldowney_1996}.
\end{remark}

\end{example}

Example~\ref{example_motivate} suggests the following much more general result:
\begin{prop}
\label{3sp1}
(i) Consider any $3 \times m$ $(0,1,-1)$ matrix $A$ and any $B \in \mathcal{Q}_0(A^t)$. Then $(AB)^{[2]}$ is a $P_0$-matrix and the nonreal spectrum of $AB$ does not intersect $\mathbb{C}_{-}$. (ii) Consider any $m \times 3$ $(0,1,-1)$ matrix $A$ and any $B \in \mathcal{Q}_0(A^t)$. Then the nonreal spectrum of $AB$ does not intersect $\mathbb{C}_{-}$.
\end{prop}
\begin{proof}
If the results hold for each $B \in \mathcal{Q}(A^t)$, they follow for each $B \in \mathcal{Q}_0(A^t)$ by closure. Choose any $B \in \mathcal{Q}(A^t)$. (i) All edge-labels in $G^{[2]}_{A,B}$ are $1$ and so $G^{[2]}_{A,B}$ is steady. On the other hand, all R-vertices in $G^{[2]}_{A,B}$ have degree less than or equal to two (three edges of the form $(12,k^l)$, $(13,k^l)$ and $(23,k^l)$ cannot all exist), and so odd intersections between cycles are impossible. Consequently, $G^{[2]}_{A,B}$ is odd$^*$. The conclusions follow from Theorem~\ref{mainP0thm} and Lemma~\ref{J2P0}. (ii) From (i), $(A^tB^t)^{[2]}$ is a $P_0$-matrix and consequently, by Lemma~\ref{J2P0}, the nonreal spectrum of $A^tB^t$ does not intersect $\mathbb{C}_{-}$. Since the nonzero spectrum of $AB$ is equal to that of $A^tB^t$ (Lemma~\ref{spectranposelemma}), the same holds for $AB$.
\end{proof}

Further generalisations of Proposition~\ref{3sp1} will be proved as Theorem~\ref{thm:3sp} later.

\begin{remark}
\label{rem:CRN3noHopf}
Proposition~\ref{3sp1} has the following interpretation in terms of the dynamics of CRNs. Consider a system of reactions where all entries in the stoichiometric matrix have magnitude $1$ (a fairly common situation) and where the kinetics satisfies the condition in (\ref{eq:CRNJac}). Suppose further there are no more than three substrates, or no more than three reactions. Then the system is incapable of Hopf bifurcations.
\end{remark}

The next result which follows with little effort, involves $(0,1,-1)$ matrices with no more than two entries in each column (or row). 

\begin{prop}
\label{prop:switchRS}
Let $A$ be an $n \times m$ $(0,1,-1)$ matrix. (i) If $A$ has no more than two nonzero entries in each column then, for each $B \in \mathcal{Q}_0(A^t)$, $AB$ and $(AB)^{[2]}$ are $P_0$-matrices and so $AB$ is positive semistable. (ii) If $A$ has no more than two nonzero entries in each row then, for any $B \in \mathcal{Q}_0(A^t)$, $AB$ is positive semistable. 
\end{prop}
\begin{proof}
If the results are true for each $B \in \mathcal{Q}(A^t)$, then they hold for each $B \in \mathcal{Q}_0(A^t)$ by closure. Choose any $B \in \mathcal{Q}(A^t)$. (i) As all edge-labels in $G_{A,B}$ and $G^{[2]}_{A,B}$ are $1$, both are steady. As all R-vertices in $G_{A,B}$ have degree less than or equal to two, by Remark~\ref{remoddintersect},  $G_{A,B}$ is in fact odd$^*$. Further, all R-vertices in $G^{[2]}_{A,B}$ have degree less than or equal to two: this follows by noting that the edge $(ij, k^j)$ exists if and only if $A_{ik}\ne 0$, which (for fixed $k$, $j$) can occur for a maximum of two indices $i$ by the assumption that no column of $A$ contains more than two nonzero entries. Consequently, $G^{[2]}_{A,B}$ is odd$^*$. The remaining conclusions follow from Theorem~\ref{mainP0thm} and Lemma~\ref{lemposstable0}. (ii) From (i), $A^tB^t$ and $(A^tB^t)^{[2]}$ are $P_0$-matrices and consequently, by Lemma~\ref{lemposstable0}, $A^tB^t$ is positive semistable. Since the nonzero spectrum of $AB$ is equal to that of $A^tB^t$ (Lemma~\ref{spectranposelemma}), the same holds for $AB$.
\end{proof}

\begin{remark}
Proposition~\ref{prop:switchRS} will be generalised as Theorem~\ref{thm:alls}, where the requirement that $A$ is a $(0,1,-1)$ matrix will be weakened to the requirement that $G_{A, A^t}$ is steady. CRNs where the degree of each S-vertex is less than or equal to two were discussed in \cite{angelileenheersontag}. 
\end{remark}

\begin{example}[Example~\ref{example_motivate0} continued]
Returning to this example we have the matrix $A$ and DSR graph $G_{A, A^t}$:
\begin{center}
\begin{tikzpicture}
[scale=0.95, place/.style={circle,draw=black!30,fill=black!0,thick,inner sep=0pt,minimum size=5.5mm},
transition/.style={rectangle,draw=black!30,fill=black!0,thick,inner sep=0pt,minimum size=5mm},
place1/.style={circle,draw=white!5,fill=white!9,thick,inner sep=0pt,minimum size=5.5mm},
transition1/.style={rectangle,draw=black!0,fill=black!0,thick,inner sep=0pt,minimum size=5mm},
pre/.style={->,shorten <=1pt,shorten >=1pt,>=stealth',semithick},
pres/.style={->,shorten <=-1pt,shorten >=-1pt,>=stealth',semithick},
pre1/.style={<-,shorten <=1pt,shorten >=1pt,>=stealth',semithick},
pres1/.style={<-,shorten <=-1pt,shorten >=0pt,>=stealth',semithick}];

\begin{scope}[xshift=-7cm]

\node at (-4.5,0) {$A = \left(\begin{array}{rccc}a&b&0\\-c&d&0\\0&e&f\\0&0&g\end{array}\right)\,,\qquad G_{A,A^t}=$};

\node[transition1] (R1) at (180:1) {$R_1$};
\node[place1] (S2) at (90:1) {$S_2$}
edge[pre1,dashed,bend right=25, -] node [auto, swap]{$\scriptstyle{c}$} (R1);
\node[transition1] (R2) at (0:1) {$R_2$}
edge[pre1,bend right=25, -] node [auto, swap]{$\scriptstyle{d}$}(S2);
\node[place1] (S1) at (-90:1) {$S_1$}
edge[pre1,bend right=25, -] node [auto, swap]{$\scriptstyle{b}$}(R2)
edge[pre1,bend left=25, -] node [auto]{$\scriptstyle{a}$}(R1);
\end{scope}

\node[place1] (S3) at (-4.7,0) {$S_3$}
edge[pre1,-] node [auto, swap]{$\scriptstyle{e}$}(R2);
\node[transition1] (R3) at (-3.3,0) {$R_3$}
edge[pre1,-] node [auto, swap]{$\scriptstyle{f}$}(S3);
\node[place1] (S4) at (-1.9,0) {$S_4$}
edge[pre1,-] node [auto, swap]{$\scriptstyle{g}$}(R3);
\end{tikzpicture}
\end{center}
$G_{A, A^t}$ is odd by inspection, and so $AB$ is a $P_0$-matrix for all $B \in \mathcal{Q}_0(A^t)$. Following removal of vertices and edges which do not participate in any cycles, we get the reduced representation of $G^{[2]}_{A,A^t}$:
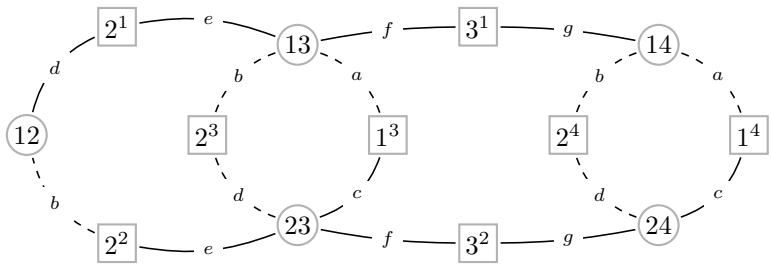
\begin{center}
\begin{tikzpicture}
[scale=1.2, place/.style={circle,draw=black!30,fill=black!0,thick,inner sep=0pt,minimum size=5.3mm},
transition/.style={rectangle,draw=black!30,fill=black!0,thick,inner sep=0pt,minimum size=5mm},
pre/.style={->,shorten <=1pt,shorten >=1pt,>=stealth',semithick},
pres/.style={->,shorten <=-1pt,shorten >=-1pt,>=stealth',semithick},
pre1/.style={<-,shorten <=1pt,shorten >=1pt,>=stealth',semithick},
pres1/.style={<-,shorten <=-1pt,shorten >=0pt,>=stealth',semithick}];

\begin{scope}[xshift=-4cm]

\node[transition] (2T3) at (180:1) {$2^3$};
\node[place] (13) at (90:1) {$13$}
edge[pre1,dashed,bend right=25, -] node [fill=white]{$\scriptstyle{b}$}  (2T3);
\node[transition] (1T3) at (0:1) {$1^3$}
edge[pre1,dashed,bend right=25, -] node [fill=white]{$\scriptstyle{a}$} (13);
\node[place] (23) at (-90:1) {$23$}
edge[pre1,bend right=25, -] node [fill=white]{$\scriptstyle{c}$} (1T3)
edge[pre1,dashed,bend left=25, -] node [fill=white]{$\scriptstyle{d}$} (2T3);
\end{scope}

\node[transition] (2T4) at (180:1) {$2^4$};
\node[place] (14) at (90:1) {$14$}
edge[pre1,dashed,bend right=25, -] node [fill=white]{$\scriptstyle{b}$}  (2T4);
\node[transition] (1T4) at (0:1) {$1^4$}
edge[pre1,dashed,bend right=25, -] node [fill=white]{$\scriptstyle{a}$}  (14);
\node[place] (24) at (-90:1) {$24$}
edge[pre1,bend right=25, -] node [fill=white]{$\scriptstyle{c}$}  (1T4)
edge[pre1,dashed,bend left=25, -] node [fill=white]{$\scriptstyle{d}$}  (2T4);

\node[transition] (3T1) at (-2,1.2) {$3^1$}
edge[pre1,bend right=5, -] node [fill=white]{$\scriptstyle{f}$}  (13)
edge[pre1,bend left=5, -] node [fill=white]{$\scriptstyle{g}$}  (14);

\node[transition] (3T2) at (-2,-1.2) {$3^2$}
edge[pre1,bend left=5, -] node [fill=white]{$\scriptstyle{f}$}  (23)
edge[pre1,bend right=5, -] node [fill=white]{$\scriptstyle{g}$}  (24);

\node[place] (12) at (-7,0) {$12$};

\node[transition] (2T1) at (-6,1.2) {$2^1$}
edge[pre1,bend left=15, -] node [fill=white]{$\scriptstyle{e}$}  (13)
edge[pre1,bend right=25, -] node [fill=white]{$\scriptstyle{d}$}  (12);
\node[transition] (2T2) at (-6,-1.2) {$2^2$}
edge[pre1,bend right=15, -] node [fill=white]{$\scriptstyle{e}$}  (23)
edge[pre1,dashed,bend left=25, -] node [fill=white]{$\scriptstyle{b}$}  (12);

\end{tikzpicture}
\end{center}
which can be computed to be odd$^*$. Consequently $(AB)^{[2]}$ is a $P_0$-matrix for all $B \in \mathcal{Q}_0(A^t)$. Thus, by Lemma~\ref{lemposstable0}, $AB$ is positive stable for each $B \in \mathcal{Q}_0(A^t)$. 
\end{example}

\begin{example}
\label{example_motivate1}
We illustrate the use of the DSR$^{[2]}$ graph to make claims about the second additive compound of a single square matrix (rather than a product of two matrices). Let $a,b,c,d,e,f,g,h,j > 0$ and consider the matrix and DSR graph:
\begin{center}
\begin{tikzpicture}
[scale=1.5, place/.style={circle,draw=black!0,fill=black!0,thick,inner sep=0pt,minimum size=5.3mm},
transition/.style={rectangle,draw=black!0,fill=black!0,thick,inner sep=0pt,minimum size=5mm},
place1/.style={circle,draw=black!0,fill=black!0,thick,inner sep=0pt,minimum size=5.5mm},
transition1/.style={rectangle,draw=black!0,fill=black!0,thick,inner sep=0pt,minimum size=5mm},
pre/.style={->,shorten <=1pt,shorten >=1pt,>=stealth',semithick},
pres/.style={->,shorten <=-1pt,shorten >=-1pt,>=stealth',semithick},
pre1/.style={<-,shorten <=1pt,shorten >=1pt,>=stealth',semithick},
pres1/.style={<-,shorten <=-1pt,shorten >=0pt,>=stealth',semithick}];

\node at (-4.5,0) {$A = \left(\begin{array}{rccc}a&b&0&0\\-c&d&g&0\\0&0&e&h\\j&0&0&f\end{array}\right)\,$};
\node at (-1.8,0) {$G_{A,I}:$};
\node[transition1] (R2) at (-135:1) {$R_2$};
\node[place1] (S2) at (-90:1) {$S_2$}
edge[pre,bend left=15, -] node [auto]{\scriptsize{$d$}} (R2);
\node[transition1] (R3) at (-45:1) {$R_3$}
edge[pre,bend left=15, ->] node [auto]{\scriptsize{$g$}}(S2);
\node[place1] (S3) at (0:1) {$S_3$}
edge[pre,bend left=15, -] node [auto]{\scriptsize{$e$}}(R3);
\node[transition1] (R4) at (45:1) {$R_4$}
edge[pre,bend left=15, ->] node [auto]{\scriptsize{$h$}}(S3);
\node[place1] (S4) at (90:1) {$S_4$}
edge[pre,bend left=15, -] node [auto]{\scriptsize{$f$}}(R4);
\node[transition1] (R1) at (135:1) {$R_1$}
edge[pre,bend left=15, ->] node [auto]{\scriptsize{$j$}}(S4)
edge[pre,dashed,bend left=20, ->] node [auto]{\scriptsize{$c$}}(S2);
\node[place1] (S1) at (180:1) {$S_1$}
edge[pre,bend left=15, -] node [auto]{\scriptsize{$a$}}(R1)
edge[pre1,bend right=15, <-] node [auto,swap]{\scriptsize{$b$}}(R2);
\end{tikzpicture}
\end{center}
Here $I$ is the $4\times 4$ identity matrix which we note lies in $\mathcal{Q}_0(A)$. $G_{A,I}$ has exactly two cycles, one of which is odd and one of which is even. Following removal of vertices and edges which do not participate in any cycles, we get the reduced representation of $G^{[2]}_{A,I}$:
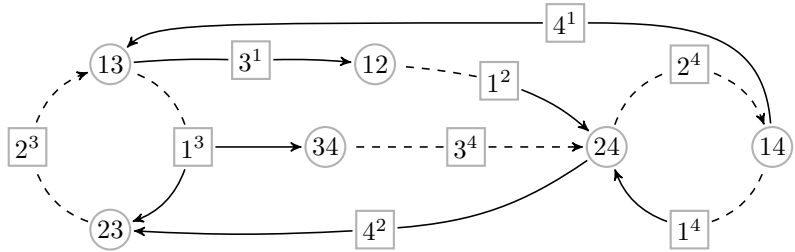
\begin{center}
\begin{tikzpicture}
[scale=1.1, place/.style={circle,draw=black!30,fill=black!0,thick,inner sep=0pt,minimum size=5.3mm},
transition/.style={rectangle,draw=black!30,fill=black!0,thick,inner sep=0pt,minimum size=5mm},
pre/.style={->,shorten <=1pt,shorten >=1pt,>=stealth',semithick},
pres/.style={->,shorten <=-1pt,shorten >=-1pt,>=stealth',semithick},
pre1/.style={<-,shorten <=1pt,shorten >=1pt,>=stealth',semithick},
pres1/.style={<-,shorten <=-1pt,shorten >=0pt,>=stealth',semithick}];

\begin{scope}[xshift=-7cm]

\node[transition] (2T3) at (180:1) {$2^3$};
\node[place] (13) at (90:1) {$13$}
edge[pre1,dashed,bend right=25, <-] (2T3);
\node[transition] (1T3) at (0:1) {$1^3$}
edge[pre1,dashed,bend right=25, -] (13);
\node[place] (23) at (-90:1) {$23$}
edge[pre1,bend right=25, <-] (1T3)
edge[pre1,dashed,bend left=25, -] (2T3);
\end{scope}

\node[transition] (2T4) at (90:1) {$2^4$};
\node[place] (14) at (0:1) {$14$}
edge[pre1,dashed,bend right=25, <-] (2T4);
\node[transition] (1T4) at (-90:1) {$1^4$}
edge[pre1,dashed,bend right=25, -] (14);
\node[place] (24) at (180:1) {$24$}
edge[pre1,bend right=25, <-] (1T4)
edge[pre1,dashed,bend left=25, -] (2T4);

\node[place] (34) at (-4.4,0) {$34$}
edge[pre1,<-] (1T3);
\node[transition] (3T4) at (-2.7,0) {$3^4$}
edge[pre1,dashed,-] (34)
edge[pre1,dashed,->] (24);

\node[transition] (3T1) at (-5.3,1.05) {$3^1$}
edge[pre1,bend right=3,-] (13);
\node[place] (12) at (-3.8,1) {$12$}
edge[pre1,bend right=3,<-] (3T1);
\node[transition] (1T2) at (-2.3,0.8) {$1^2$}
edge[pre1,dashed,bend right=3,-] (12)
edge[pre1,bend left=10,->] (24);

\node[transition] (4T2) at (-3.8,-1) {$4^2$}
edge[pre1,bend left=3,->] (23)
edge[pre1,bend right=15,-] (24);

\node[transition] (4T1) at (-1.5,1.5) {$4^1$};

\draw[pre1,->] (4T1) .. controls (-6,1.5) and (-6.5, 1.5) .. (13);
\draw[pre1,-] (4T1) .. controls (0,1.5) and (0.9, 1.4) .. (14);

\end{tikzpicture}
\end{center}
Edge labels have been omitted for the reason that $G^{[2]}_{A, I}$ is odd, namely there are no e-cycles at all. Consequently $A^{[2]} = (AI)^{[2]}$ is a $P_0$-matrix, and $A$ cannot have a pair of nonreal eigenvalues in $\mathbb{C}_{-}$.
\end{example}

\begin{remark}\label{rem:monotonecycles} 
The matrix $A$ in Example~\ref{example_motivate1} is not obviously amenable to analysis using the results of \cite{gouze98} or \cite{angelihirschsontag}: the sufficient conditions for the absence of {\em stable} oscillation in dynamical systems developed in these papers are neither necessary nor sufficient for the preclusion of Hopf bifurcation. 
\end{remark}

\end{section}

\begin{section}{Relationships between the DSR and DSR$^{[2]}$ graphs}
\label{sec:theory}

Throughout this section we consider two matrices $A\in\mathbb{R}^{n \times m}$ and $B \in \mathbb{R}^{m \times n}$ ($n \geq 2$). In order to understand better the DSR$^{[2]}$ graph, it is convenient to develop some theory on the relationship between structures in $G_{A, B}$ and $G^{[2]}_{A, B}$. The results presented below begin this process. Applications of this theory are presented in the final section.

\begin{subsection}{Removing leaves}

Our first observation is that a certain simple operation on DSR graphs of the form $G_{A, A^t}$ -- the removal of pendant R-vertices and their incident edges -- does not alter the cycle structure of the corresponding DSR$^{[2]}$ graph. This claim does not extend to the removal of pendant S-vertices. 

\begin{lemma}
\label{lem:leafremove}
Consider an $n \times m$ matrix $A$ ($m \geq 2$) whose $k$th column has a single nonzero entry $A_{k_0,k}$, so that vertex $R_k$ in $G_{A, A^t}$ has degree $1$. Let $\tilde A$ be the submatrix of $A$ where the $k$th column has been removed. Then $G^{[2]}_{\tilde A, \tilde A^t}$ is odd$^*$ (resp. odd, resp. steady) if and only if $G^{[2]}_{A, A^t}$ is odd$^*$ (resp. odd, resp. steady).
\end{lemma}
\begin{proof}
In one direction the result is trivial: if $G^{[2]}_{A, A^t}$ is odd$^*$ (resp. odd, resp. steady), then so is $G^{[2]}_{\tilde A, \tilde A^t}$ as it is a subgraph of $G^{[2]}_{A, A^t}$.

In the other direction, we show that $G^{[2]}_{\tilde A, \tilde A^t}$ is an induced subgraph of $G^{[2]}_{A, A^t}$ obtained by removing only leaves of $G^{[2]}_{A, A^t}$, and consequently has the same cycle structure. Note that $G^{[2]}_{\tilde A, \tilde A^t}$ is the induced subgraph of $G^{[2]}_{A, A^t}$ obtained by removing all R-vertices of the form $k^j$ and their incident edges. However, an edge in $G^{[2]}_{A, A^t}$ of the form $(ij, k^j)$ exists if and only if $A_{ik} \neq 0$ which occurs iff $i = k_0$. Thus for each $j$, column $k^j$ of $\overline{\mathbf{L}}^A$ has a single nonzero entry (in row $k_0j$), and the corresponding R-vertex in $G^{[2]}_{A, A^t}$ is a leaf which can be removed without affecting any cycles of $G^{[2]}_{A, A^t}$.
\end{proof}

\end{subsection}

\begin{subsection}{Projections of DSR$^{[2]}$ cycles} 

It is useful to distinguish two types of edges in a DSR$^{[2]}$ graph, $(ij, k^{\min\{i,j\}})$ and $(ij, k^{\max\{i,j\}})$. Recall from Section \ref{sec:C2decomp} that if $i<j$ then  
\[\overline{\mathbf L}^A_{ij,k^i}=A_{jk}\,, \quad\underline{\mathbf L}^{B}_{k^i,ij}=B_{kj},\quad \mbox{whereas} \quad \overline{\mathbf L}^A_{ij,k^j}=-A_{ik}\,, \quad\underline{\mathbf L}^{B}_{k^j,ij}=-B_{ki}\,.
\]

This fact, in conjunction with the definition of $G_{A, B}$, implies that for any edge $(ij, k^i)$ in $G^{[2]}_{A, B}$ there exists an edge $(S_j, R_k)$ in $G_{A, B}$ having the same direction, the same label and the same sign; and similarly,  for any edge $(ij, k^j)$ in $G^{[2]}_{A, B}$ there exists an edge $(S_i, R_k)$ in $G_{A,B}$ having the same direction and label, but {\em opposite sign}. We have defined a mapping of the edge-set of $G^{[2]}_{A, B}$ into the edge-set of $G_{A, B}$.

\begin{definition1}\label{def:proj}
Let $E = E(G_{A, B})$ and $E2 = E(G^{[2]}_{A, B})$ denote the set of edges of $G_{A, B}$ and $G^{[2]}_{A, B}$ respectively.  Similarly let $V = V(G_{A, B})$ and $V2 = V(G^{[2]}_{A, B})$. The mapping $\pi:E2\to E$ defined by
\begin{equation*}
\pi((ij, k^j))=(S_i, R_k)\,\,(\mbox{and  } \pi((ij, k^i))=(S_j, R_k))
\end{equation*} 
is termed the {\em projection map} on edges of {\em DSR}$^{[2]}.$
\end{definition1}

\begin{remark}\label{rem:pi} A few immediate properties of the projection map are worth emphasising.

1. $\pi$ is surjective if $n\ge 2$, but not necessarily injective if $n\ge 3$. Indeed, if $(S_i, R_k)\in E$ then for any $j\in \{1,\ldots,n\}\backslash \{i\}$ we have $\pi((ij, k^j))=(S_i, R_k)$. If there is more than one choice for such $j$ then $\pi$ is not injective. 

2. Following the discussion preceding Definition \ref{def:proj}, $\pi$ acts on edge-signs as follows: 
\begin{equation}\label{eq:projsign}
\mathrm{sign}(\pi(e)) = \left\{\begin{array}{rl}
\mathrm{sign}(e) & \mbox{ if } e=(ij,k^{\min\{i,j\}})\\
- \mathrm{sign}(e) & \mbox{ if } e=(ij,k^{\max\{i,j\}})
\end{array}\right.
\end{equation}   

3. $\pi$ does not extend naturally to vertices of $G^{[2]}_{A, B}$. While $\pi$ may be viewed as mapping the R-vertex $k^l \in V2$ to $R_k \in V$, this interpretation is ambiguous for $S$-vertices $ij\in V2$: edges in $E2$ incident to $ij$ may be mapped to edges in $E$ incident to either $S_i$ or to $S_j$. As a consequence, a walk in $G^{[2]}_{A, B}$ does not necessarily project to a walk in $G_{A, B}$. This is illustrated in the following diagram, which also shows that projection may change edge signs.
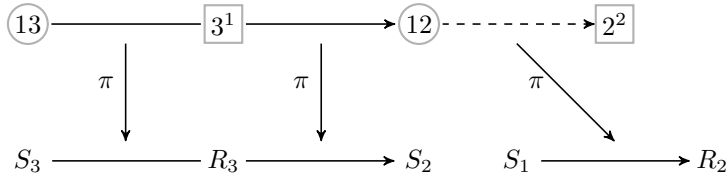
\begin{center}
\begin{tikzpicture}
[scale=2.6, place/.style={circle,draw=black!30,fill=black!0,thick,inner sep=0pt,minimum size=5.3mm},
transition/.style={rectangle,draw=black!30,fill=black!0,thick,inner sep=0pt,minimum size=5mm},
place1/.style={circle,draw=white!5,fill=white!9,thick,inner sep=0pt,minimum size=5.5mm},
transition1/.style={rectangle,draw=black!0,fill=black!0,thick,inner sep=0pt,minimum size=5mm},
pre/.style={->,shorten <=1pt,shorten >=1pt,>=stealth',semithick},
pre1/.style={<-,shorten <=1pt,shorten >=1pt,>=stealth',semithick},
post/.style={shorten >=1pt,shorten >=1pt,>=stealth',semithick}];
\node[place] (13) at (-2,0) {$13$};
\node[place] (12) at (0,0) {$12$};
\node[transition] ([11])  at (-1,0)  {$3^1$}
edge[pre1,-]  (13)
edge [pre] (12);
\node[transition] ([21]) at (1,0)  {$2^2$}
edge [<-, dashed, post] (12);

\node[place1] (13) at (-2,-.7) {$S_3$};
\node[place1] (12) at (0,-.7) {$S_2$};
\node[transition1] ([11])  at (-1,-.7)  {$R_3$}
edge[pre1,-] (13)
edge [pre] (12);

\draw[>=stealth',semithick,->] (-1.5,-.1) -- (-1.5,-.6);
\node at (-1.6, -.3) {$\pi$};
\draw[>=stealth',semithick,->] (-.5,-.1) -- (-.5,-.6);
\node at (-.6, -.3) {$\pi$};
\draw[>=stealth',semithick,->] (.5,-.1) -- (1,-.6);
\node at (.6, -.3) {$\pi$};

\node[place1](14) at (.5,-.7) {$S_1$};
\node[transition1] ([14])  at (1.5,-.7)  {$R_2$}
edge[->,pre1] (14);
\end{tikzpicture}
\end{center}
\end{remark}   
Having defined the projection map, it is natural to examine the projection of cycles in $G^{[2]}_{A, B}$.

\begin{prop}\label{prop:proj}
Let $C=(e_1,\ldots, e_N)$ be a cycle in $G^{[2]}_{A, B}$ with edges enumerated so that the initial vertex of $e_1$ is an S-vertex $ij$.
Then there exists a partition of $\{1,\ldots,N\}=\{i_1<\ldots< i_M\}\cup\{j_1<\ldots< j_{N-M}\}$ such that $W'=(\pi(e_{i_1}),\ldots,\pi(e_{i_M}))$ and $W''=(\pi(e_{j_1}),\ldots,\pi(e_{j_{N-M}}))$ satisfy either:

(i) $W'$ and $W''$ are closed walks, one containing $S_i$ and the other containing $S_j;$ or

(ii) $W'$ and $W''$ are walks, one from $S_i$ to $S_j$ and the other from $S_j$ to $S_i$.  
\end{prop}

\begin{proof} Let $(ij=a_1b_1, \ldots, a_{N/2}b_{N/2}, a_{N/2+1}b_{N/2+1}=ij)$ denote the S-vertices of $C$, and such that $a_i<b_i$ for each $i$. The construction of $W'$ and $W''$ is done by successively appending the walks of length two  $(\pi(e_{1}), \pi(e_{2})), (\pi(e_{3}), \pi(e_{4})),\ldots, (\pi(e_{N-1}), \pi(e_{N}))$ of $G_{A,B}$ to either $W'$ or $W''$. Begin by defining $W'_{1}$ and $W''_{1}$ as the empty one-vertex walks $\{S_i\}$ and $\{S_j\}$, respectively. Next, inductively define, for $r=2,\ldots,N/2+1$, walks $W'_{r}$ and $W''_{r}$ in $G_{A,B}$ whose end vertices are different elements of $\{S_{a_r}, S_{b_r}\}$ in the following fashion. Assume that $S_{a_{r-1}}$ is the end vertex of $W'_{r-1}$ and $S_{b_{r-1}}$ is the end vertex of $W''_{r-1}$ (the other case is similar). The walk $w_r=(e_{2r-3},e_{2r-2})$ has the form $(a_{r-1}b_{r-1}, k^l, a_rb_r)$, where one of the following possibilities holds: (1) $l=a_{r-1}=a_r$ so that $\pi(w_r) = (S_{b_{r-1}},R_k,S_{b_r})$; (2) $l=a_{r-1}=b_r$ so that $\pi(w_r) = (S_{b_{r-1}},R_k,S_{a_r})$; (3) $l=b_{r-1}=a_r$, so that $\pi(w_r) = (S_{a_{r-1}},R_k,S_{b_r})$; (4) $l=b_{r-1}=b_r$ so that $\pi(w_r) = (S_{a_{r-1}},R_k,S_{a_r})$. In cases (1) and (2) define $W''_r=(W''_{r-1},\pi(w_r))$ and let  $W'_{r}=W'_{r-1}$; in cases (3) and (4) define $W'_r=(W'_{r-1},\pi(w_r))$ and let  $W''_{r}=W''_{r-1}$. Either way, $W'_r$ and $W''_r$ are walks ending at different vertices from $\{S_{a_r}, S_{b_r}\}$. It is easy to check that $W'=W'_{N/2+1}$ and $W''=W''_{N/2+1}$ satisfy the conclusion of the proposition.
\end{proof}

\begin{remark}\label{rem:W'W''}
1. In Case (i) of Proposition \ref{prop:proj} $W'$ and $W''$ are closed walks in $G_{A, B}$ and it is not hard to see that choosing a different initial vertex from $C$ in $G^{[2]}_{A, B}$ leads to the same two closed walks (albeit possibly renamed, and traversed from different initial points); in Case (ii) of Proposition \ref{prop:proj} $W'$ and $W''$ traversed successively form a closed walk of length $N$, which we will denote by $W'\sqcup W''.$ Again choosing a different initial vertex from $C$ leads to the same object, traversed from a different initial point.

2. By construction, a pair of consecutive edges of $C$ beginning and terminating at an S-vertex are projected to a pair of consecutive edges in either $W'$ or $W''$. With notation from Proposition \ref{prop:proj}, this means that $i_{2r}=i_{2r-1}+1$ for each $r = 1, \ldots, M/2$, and $j_{2r}=j_{2r-1}+1$ for each $r = 1, \ldots, (N-M)/2$.  In particular, a sequence $(S_a, R_k, S_a)$ traversing an unoriented edge $(S_a, R_k)$ in each direction consecutively cannot occur in $W', W''$ or $W'\sqcup W''$: such a sequence must be the projection of a sequence of edges in $G^{[2]}_{A, B}$ of the form $(ab, k^b, ab)$, traversing an unoriented edge $(ab, k^b)$ back and forth, which clearly cannot occur in a cycle $C$.
\end{remark}

It is transparent from the previous remark that whether projection of a cycle $C$ in $G^{[2]}_{A,B}$ leads to Case (i) or Case (ii) in Proposition \ref{prop:proj} is a property of $C$ alone. A cycle $C$ of $G^{[2]}_{A,B}$ will be called {\bf direct} if it projects as in Case (i) and {\bf twisted} if it projects as in Case (ii). In view of Remark \ref{rem:W'W''} the following notion is well-defined: 

\begin{definition1}\label{def:projcycle}
Let $C$ be a cycle in $G^{[2]}_{A, B}$ and $W',W''$ corresponding walks in $G_{A, B}$ constructed as in Proposition \ref{prop:proj}. The {\em projection} $\pi(C)$ of $C$ is defined as follows:

(i) if $C$ is direct then $\pi(C)=\{W',W''\};$

(ii) if $C$ is twisted then $\pi(C)=W'\sqcup W''$.
\end{definition1}

\begin{example}
\label{example:directtwisted}
We return to the DSR$^{[2]}$ and DSR graphs in Example~\ref{example_motivate0} to illustrate the projections of direct and twisted cycles. The cycle
\[
(13, 2^3, 23, 3^2, 24, 1^4, 14, 3^1, 13)
\]
(with image shaded in the DSR$^{[2]}$ graph on the left) is direct, and projects to the two closed walks
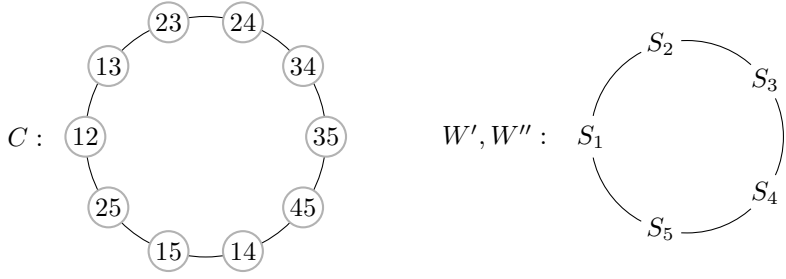
\[
W' = (S_1, R_2, S_2, R_1, S_1), \quad
W'' = (S_3, R_3, S_4, R_3, S_3)\,
\]
(with images shaded in the DSR graph on the right).
\begin{center}
\begin{tikzpicture}
[scale=0.93, place/.style={circle,draw=black!30,fill=black!0,thick,inner sep=0pt,minimum size=5.3mm},
transition/.style={rectangle,draw=black!30,fill=black!0,thick,inner sep=0pt,minimum size=5mm},
pre/.style={->,shorten <=1pt,shorten >=1pt,>=stealth',semithick},
pres/.style={->,shorten <=-1pt,shorten >=-1pt,>=stealth',semithick},
pre1/.style={<-,shorten <=1pt,shorten >=1pt,>=stealth',semithick},
pres1/.style={<-,shorten <=-1pt,shorten >=0pt,>=stealth',semithick}];

\begin{scope}[xshift=-4cm]

\node[transition] (2T3) at (180:1) {$2^3$};
\node[place] (13) at (90:1) {$13$}
edge[pre1,dashed,bend right=25, -] node [fill=white]{$\scriptstyle{b}$}  (2T3)
edge[color=black!40,opacity=0.5,line width=0.2cm,bend right=25, -] (2T3);
\node[transition] (1T3) at (0:1) {$1^3$}
edge[pre1,dashed,bend right=25, -] node [fill=white]{$\scriptstyle{a}$} (13);
\node[place] (23) at (-90:1) {$23$}
edge[pre1,bend right=25, -] node [fill=white]{$\scriptstyle{c}$} (1T3)
edge[pre1,dashed,bend left=25, -] node [fill=white]{$\scriptstyle{d}$} (2T3)
edge[color=black!40,opacity=0.5,line width=0.2cm,bend left=25, -] (2T3);
\end{scope}

\node[transition] (2T4) at (180:1) {$2^4$};
\node[place] (14) at (90:1) {$14$}
edge[pre1,dashed,bend right=25, -] node [fill=white]{$\scriptstyle{b}$}  (2T4);
\node[transition] (1T4) at (0:1) {$1^4$}
edge[pre1,dashed,bend right=25, -] node [fill=white]{$\scriptstyle{a}$}  (14)
edge[color=black!40,opacity=0.5,line width=0.2cm,bend right=25, -] (14);
\node[place] (24) at (-90:1) {$24$}
edge[pre1,bend right=25, -] node [fill=white]{$\scriptstyle{c}$}  (1T4)
edge[color=black!40,opacity=0.5,line width=0.2cm,bend right=25, -] (1T4)
edge[pre1,dashed,bend left=25, -] node [fill=white]{$\scriptstyle{d}$}  (2T4);

\node[transition] (3T1) at (-2,1.2) {$3^1$}
edge[pre1,bend right=5, -] node [fill=white]{$\scriptstyle{f}$}  (13)
edge[color=black!40,opacity=0.5,line width=0.2cm,bend right=5, -] (13)
edge[pre1,bend left=5, -] node [fill=white]{$\scriptstyle{g}$}  (14)
edge[color=black!40,opacity=0.5,line width=0.2cm,bend left=5, -] (14);

\node[transition] (3T2) at (-2,-1.2) {$3^2$}
edge[pre1,bend left=5, -] node [fill=white]{$\scriptstyle{f}$}  (23)
edge[color=black!40,opacity=0.5,line width=0.2cm,bend left=5, -] (23)
edge[pre1,bend right=5, -] node [fill=white]{$\scriptstyle{g}$}  (24)
edge[color=black!40,opacity=0.5,line width=0.2cm,bend right=5, -] (24);

\node[place] (12) at (-7,0) {$12$};

\node[transition] (2T1) at (-6,1.2) {$2^1$}
edge[pre1,bend left=15, -] node [fill=white]{$\scriptstyle{e}$}  (13)
edge[pre1,bend right=25, -] node [fill=white]{$\scriptstyle{d}$}  (12);
\node[transition] (2T2) at (-6,-1.2) {$2^2$}
edge[pre1,bend right=15, -] node [fill=white]{$\scriptstyle{e}$}  (23)
edge[pre1,dashed,bend left=25, -] node [fill=white]{$\scriptstyle{b}$}  (12);

\end{tikzpicture}
\hspace{1cm}
\begin{tikzpicture}
[scale=0.75, place/.style={circle,draw=black!7,fill=black!0,thick,inner sep=0pt,minimum size=5.5mm},
transition/.style={rectangle,draw=black!7,fill=black!0,thick,inner sep=0pt,minimum size=5mm},
place1/.style={circle,draw=white!5,fill=white!9,thick,inner sep=0pt,minimum size=5.5mm},
transition1/.style={rectangle,draw=black!0,fill=black!0,thick,inner sep=0pt,minimum size=5mm},
pre/.style={->,shorten <=1pt,shorten >=1pt,>=stealth',semithick},
pres/.style={->,shorten <=-1pt,shorten >=-1pt,>=stealth',semithick},
pre1/.style={<-,shorten <=1pt,shorten >=1pt,>=stealth',semithick},
pres1/.style={<-,shorten <=-1pt,shorten >=0pt,>=stealth',semithick}];

\begin{scope}[xshift=-7cm]

\node[transition1] (R1) at (180:1) {$R_1$};
\node[place1] (S2) at (90:1) {$S_2$}
edge[pre1,dashed,bend right=25, -] node [auto, swap]{$\scriptstyle{c}$} (R1)
edge[color=black!40,opacity=0.5,line width=0.2cm,bend right=25, -] (R1);
\node[transition1] (R2) at (0:1) {$R_2$}
edge[pre1,bend right=25, -] node [auto, swap]{$\scriptstyle{d}$}(S2)
edge[color=black!40,opacity=0.5,line width=0.2cm,bend right=25, -] (S2);
\node[place1] (S1) at (-90:1) {$S_1$}
edge[pre1,bend right=25, -] node [auto, swap]{$\scriptstyle{b}$}(R2)
edge[color=black!40,opacity=0.5,line width=0.2cm,bend right=25, -] (R2)
edge[pre1,bend left=25, -] node [auto]{$\scriptstyle{a}$}(R1)
edge[color=black!40,opacity=0.5,line width=0.2cm,bend left=25, -] (R1);
\end{scope}

\node[place1] (S3) at (-4.7,0) {$S_3$}
edge[pre1,-] node [auto, swap]{$\scriptstyle{e}$}(R2);
\node[transition1] (R3) at (-3.3,0) {$R_3$}
edge[pre1,-] node [auto, swap]{$\scriptstyle{f}$}(S3)
edge[color=black!40,opacity=0.5,line width=0.2cm,-] (S3);
\node[place1] (S4) at (-1.9,0) {$S_4$}
edge[pre1,-] node [auto, swap]{$\scriptstyle{g}$}(R3)
edge[color=black!40,opacity=0.5,line width=0.2cm,-] (R3);
\end{tikzpicture}
\end{center}
On the other hand, the cycle
\[
(13, 1^3, 23, 2^2, 12, 2^1, 13)
\]
(with image shaded in the DSR$^{[2]}$ graph on the left) is twisted, and projects to the closed walk
\[
W'\sqcup W'' = (S_1, R_1, S_2, R_2, S_3, R_2, S_1)\,
\]
(with image shaded in the DSR graph on the right).
\begin{center}
\begin{tikzpicture}
[scale=0.93, place/.style={circle,draw=black!30,fill=black!0,thick,inner sep=0pt,minimum size=5.3mm},
transition/.style={rectangle,draw=black!30,fill=black!0,thick,inner sep=0pt,minimum size=5mm},
pre/.style={->,shorten <=1pt,shorten >=1pt,>=stealth',semithick},
pres/.style={->,shorten <=-1pt,shorten >=-1pt,>=stealth',semithick},
pre1/.style={<-,shorten <=1pt,shorten >=1pt,>=stealth',semithick},
pres1/.style={<-,shorten <=-1pt,shorten >=0pt,>=stealth',semithick}];

\begin{scope}[xshift=-4cm]

\node[transition] (2T3) at (180:1) {$2^3$};
\node[place] (13) at (90:1) {$13$}
edge[pre1,dashed,bend right=25, -] node [fill=white]{$\scriptstyle{b}$}  (2T3);
\node[transition] (1T3) at (0:1) {$1^3$}
edge[pre1,dashed,bend right=25, -] node [fill=white]{$\scriptstyle{a}$} (13)
edge[color=black!40,opacity=0.5,line width=0.2cm,bend right=25, -] (13);
\node[place] (23) at (-90:1) {$23$}
edge[pre1,bend right=25, -] node [fill=white]{$\scriptstyle{c}$} (1T3)
edge[color=black!40,opacity=0.5,line width=0.2cm,bend right=25, -] (1T3)
edge[pre1,dashed,bend left=25, -] node [fill=white]{$\scriptstyle{d}$} (2T3);
\end{scope}

\node[transition] (2T4) at (180:1) {$2^4$};
\node[place] (14) at (90:1) {$14$}
edge[pre1,dashed,bend right=25, -] node [fill=white]{$\scriptstyle{b}$}  (2T4);
\node[transition] (1T4) at (0:1) {$1^4$}
edge[pre1,dashed,bend right=25, -] node [fill=white]{$\scriptstyle{a}$}  (14);
\node[place] (24) at (-90:1) {$24$}
edge[pre1,bend right=25, -] node [fill=white]{$\scriptstyle{c}$}  (1T4)
edge[pre1,dashed,bend left=25, -] node [fill=white]{$\scriptstyle{d}$}  (2T4);

\node[transition] (3T1) at (-2,1.2) {$3^1$}
edge[pre1,bend right=5, -] node [fill=white]{$\scriptstyle{f}$}  (13)
edge[pre1,bend left=5, -] node [fill=white]{$\scriptstyle{g}$}  (14);

\node[transition] (3T2) at (-2,-1.2) {$3^2$}
edge[pre1,bend left=5, -] node [fill=white]{$\scriptstyle{f}$}  (23)
edge[pre1,bend right=5, -] node [fill=white]{$\scriptstyle{g}$}  (24);

\node[place] (12) at (-7,0) {$12$};

\node[transition] (2T1) at (-6,1.2) {$2^1$}
edge[pre1,bend left=15, -] node [fill=white]{$\scriptstyle{e}$}  (13)
edge[color=black!40,opacity=0.5,line width=0.2cm,bend left=15, -] (13)
edge[pre1,bend right=25, -] node [fill=white]{$\scriptstyle{d}$}  (12)
edge[color=black!40,opacity=0.5,line width=0.2cm,bend right=25, -] (12);
\node[transition] (2T2) at (-6,-1.2) {$2^2$}
edge[pre1,bend right=15, -] node [fill=white]{$\scriptstyle{e}$}  (23)
edge[color=black!40,opacity=0.5,line width=0.2cm,bend right=15, -] (23)
edge[pre1,dashed,bend left=25, -] node [fill=white]{$\scriptstyle{b}$}  (12)
edge[color=black!40,opacity=0.5,line width=0.2cm,bend left=25, -] (12);

\end{tikzpicture}
\hspace{1cm}
\begin{tikzpicture}
[scale=0.75, place/.style={circle,draw=blue!50,fill=blue!5,thick,inner sep=0pt,minimum size=5.5mm},
transition/.style={rectangle,draw=black!0,fill=black!0,thick,inner sep=0pt,minimum size=5mm},
place1/.style={circle,draw=white!5,fill=white!9,thick,inner sep=0pt,minimum size=5.5mm},
transition1/.style={rectangle,draw=black!0,fill=black!0,thick,inner sep=0pt,minimum size=5mm},
pre/.style={->,shorten <=1pt,shorten >=1pt,>=stealth',semithick},
pres/.style={->,shorten <=-1pt,shorten >=-1pt,>=stealth',semithick},
pre1/.style={<-,shorten <=1pt,shorten >=1pt,>=stealth',semithick},
pres1/.style={<-,shorten <=-1pt,shorten >=0pt,>=stealth',semithick}];

\begin{scope}[xshift=-7cm]

\node[transition1] (R1) at (180:1) {$R_1$};
\node[place1] (S2) at (90:1) {$S_2$}
edge[pre1,dashed,bend right=25, -] node [auto, swap]{$\scriptstyle{c}$} (R1)
edge[color=black!40,opacity=0.5,line width=0.2cm,bend right=25, -] (R1);
\node[transition1] (R2) at (0:1) {$R_2$}
edge[pre1,bend right=25, -] node [auto, swap]{$\scriptstyle{d}$}(S2)
edge[color=black!40,opacity=0.5,line width=0.2cm,bend right=25, -] (S2);
\node[place1] (S1) at (-90:1) {$S_1$}
edge[pre1,bend right=25, -] node [auto, swap]{$\scriptstyle{b}$}(R2)
edge[color=black!40,opacity=0.5,line width=0.2cm,bend right=25, -] (R2)
edge[pre1,bend left=25, -] node [auto]{$\scriptstyle{a}$}(R1)
edge[color=black!40,opacity=0.5,line width=0.2cm,bend left=25, -] (R1);
\end{scope}

\node[place1] (S3) at (-4.7,0) {$S_3$}
edge[pre1,-] node [auto, swap]{$\scriptstyle{e}$}(R2)
edge[color=black!40,opacity=0.5,line width=0.2cm,-] (R2);
\node[transition1] (R3) at (-3.3,0) {$R_3$}
edge[pre1,-] node [auto, swap]{$\scriptstyle{f}$}(S3);
\node[place1] (S4) at (-1.9,0) {$S_4$}
edge[pre1,-] node [auto, swap]{$\scriptstyle{g}$}(R3);
\end{tikzpicture}
\end{center}
\end{example}

\begin{remark}\label{rem:extendpi}
 Let $C$ be a cycle in $G^{[2]}_{A, B}$.

1. If $C$ is twisted then $|\pi(C)|=|C|,$ whereas if $C$ is direct and $\pi(C)=\{W',W''\}$ then $|C|=|W'|+|W''|$. 

2. In the case that $C$ is direct, various scenarios are possible: one of $W'$ or $W''$ may be empty; $W'$ and $W''$ may be disjoint; they may share vertices or edges; or they might even coincide. To see how one of $W'$ or $W''$ may be empty, consider the following example. Let $A \in \mathbb{R}^{5 \times 4}$ be such that $A_{11},A_{14},A_{31},A_{32},A_{42},A_{43},A_{53},A_{54}$ are positive. We get an 8-cycle $C = (23,2^2,24,3^2,25,4^2,12,1^2,23)$ in $G^{[2]}_{A, A^t}$ which projects to an empty walk $W'$ and an 8-cycle $W''=(S_3,R_2,S_4,R_3,S_5,R_4,S_1,R_1,S_3)$ in $G_{A, A^t}$:
\begin{center}
\begin{tikzpicture}
[scale=1.2, place/.style={circle,draw=black!30,fill=black!0,thick,inner sep=0pt,minimum size=5.3mm},
transition/.style={rectangle,draw=black!30,fill=black!0,thick,inner sep=0pt,minimum size=5mm},
place1/.style={circle,draw=white!5,fill=white!9,thick,inner sep=0pt,minimum size=5.3mm},
transition1/.style={rectangle,draw=black!0,fill=black!0,thick,inner sep=0pt,minimum size=5mm},
pre/.style={->,shorten <=1pt,shorten >=1pt,>=stealth',semithick},
pres/.style={->,shorten <=-1pt,shorten >=-1pt,>=stealth',semithick},
pre1/.style={<-,shorten <=1pt,shorten >=1pt,>=stealth',semithick},
pres1/.style={<-,shorten <=-1pt,shorten >=0pt,>=stealth',semithick}];
\node at (-2,0) {$C:$};
\node[place] (12) at (-1,1) {$12$};
\node[place] (23) at (1,1) {$23$};
\node[place] (25) at (-1,-1) {$25$};
\node[place] (24) at (1,-1) {$24$};
\node[transition] ([12])  at (0,1)  {$1^2$}
edge[pre,-, dashed] node [auto, swap]{\scriptsize$A_{11}$} (12)
edge [pre,-] node [auto]{\scriptsize$A_{31}$} (23);
\node[transition] ([22])  at (1,0)  {$2^2$}
edge[pre,-] node [auto, swap]{\scriptsize$A_{32}$} (23)
edge [pre,-] node [auto]{\scriptsize$A_{42}$} (24);
\node[transition] ([32])  at (0,-1)  {$3^2$}
edge[pre,-] node [auto, swap]{\scriptsize$A_{43}$} (24)
edge [pre,-] node [auto]{\scriptsize$A_{53}$} (25);
\node[transition] ([42])  at (-1,0)  {$4^2$}
edge[pre,-] node [auto, swap]{\scriptsize$A_{54}$} (25)
edge [pre,-, dashed] node [auto]{\scriptsize$A_{14}$} (12);

\node at (3,0) {$W':\varnothing$};

\begin{scope}[xshift=6.5cm]
\node at (-2,0) {$W'':$};
\node[place1] (12) at (-1,1) {$S_1$};
\node[place1] (23) at (1,1) {$S_3$};
\node[place1] (25) at (-1,-1) {$S_5$};
\node[place1] (24) at (1,-1) {$S_4$};
\node[transition1] ([12])  at (0,1)  {$R_1$}
edge[pres,-] node [auto, swap]{\scriptsize$A_{11}$} (12)
edge [pres,-] node [auto]{\scriptsize$A_{31}$} (23);
\node[transition1] ([22])  at (1,0)  {$R_2$}
edge[pres,-] node [auto, swap]{\scriptsize$A_{32}$} (23)
edge [pres,-] node [auto]{\scriptsize$A_{42}$} (24);
\node[transition1] ([32])  at (0,-1)  {$R_3$}
edge[pres,-] node [auto, swap]{\scriptsize$A_{43}$} (24)
edge [pres,-] node [auto]{\scriptsize$A_{53}$} (25);
\node[transition1] ([42])  at (-1,0)  {$R_4$}
edge[pres,-] node [auto, swap]{\scriptsize$A_{54}$} (25)
edge [pres,-] node [auto]{\scriptsize$A_{14}$} (12);
\end{scope}
\end{tikzpicture}
\end{center}
Note that in this case, the index 2 is common to all S-vertices in $C$, and, disregarding signs of edges, the edge-sequence of $W''$ mirrors that of $C$. In situations such as this, we make an abuse of terminology and write $\pi(C)=W''$ rather than $\pi(C)=\{\varnothing, W''\}.$ To see how $W'$ and $W''$ may coincide consider the 20-cycle $C$ illustrated below. For simplicity only S-vertices are depicted:
\begin{center}
\begin{tikzpicture}
[scale=1.6, place/.style={circle,draw=black!30,fill=black!0,thick,inner sep=0pt,minimum size=5.3mm},
transition/.style={rectangle,draw=black!50,fill=black!5,thick,inner sep=0pt,minimum size=5mm},
place1/.style={circle,draw=white!5,fill=white!9,thick,inner sep=0pt,minimum size=5.3mm},
pre/.style={->,shorten <=1pt,shorten >=1pt,>=stealth',semithick},
pres/.style={->,shorten <=-1pt,shorten >=-1pt,>=stealth',semithick},
pre1/.style={<-,shorten <=1pt,shorten >=1pt,>=stealth',semithick},
pres1/.style={<-,shorten <=-1pt,shorten >=0pt,>=stealth',semithick}];

\node at (-1.5,0) {$C:$};
\draw (0,0) circle (1cm);
\node[place] (12) at (180:1) {$12$};
\node[place] (13) at (144:1) {$13$};
\node[place] (23) at (108:1) {$23$};
\node[place] (24) at (72:1) {$24$};
\node[place] (34) at (36:1) {$34$};
\node[place] (35) at (0:1) {$35$};
\node[place] (45) at (-36:1) {$45$};
\node[place] (14) at (-72:1) {$14$};
\node[place] (15) at (-108:1) {$15$};
\node[place] (25) at (-144:1) {$25$};

\begin{scope}[xshift=4cm, scale=.8]
\node at (-2,0) {$W',W'':$};
\draw (0,0) circle (1cm);
\node[place1] (12) at (180:1) {$S_1$};
\node[place1] (13) at (106:1) {$S_2$};
\node[place1] (23) at (36:1) {$S_3$};
\node[place1] (24) at (-36:1) {$S_4$};
\node[place1] (34) at (-106:1) {$S_5$};
\end{scope}
\end{tikzpicture}
\end{center}
\end{remark}

{\bf Identifying whether a cycle is direct or twisted.} If $ij$ and $i'j'$ are two consecutive S-vertices of a cycle $C$ in $G^{[2]}_{A,B}$, then $\{i,j\}\cap\{i',j'\}$ must include at least one element, say $i$. We call the pair $(ij, ij')$ an {\bf inversion} if $(i-j)(i-j') < 0$. Whether a cycle is direct or twisted can be computed by counting inversions:

\begin{prop}\label{prop:parityinv}
A cycle $C$ of a DSR$^{[2]}$ graph $G^{[2]}_{A,B}$ is direct if it has an even number of inversions, and twisted otherwise.
\end{prop}

\begin{proof}
Let $C$ be a direct cycle with the sequence of S-vertices $(ij=a_1b_1, a_2b_2,\ldots, a_rb_r=ij)$. Assume that $a_k<b_k$ for all $k$. By the construction in Proposition~\ref{prop:proj}, one element of $\{a_k,b_k\}$ is assigned to $W'$ and the other is assigned to $W''$. Define $\mathrm{ind}(a_kb_k)$ as
$$
\mathrm{ind}(a_kb_k)=
\begin{cases} 
1, &\mbox{if } a_k \mbox{ is assigned to } W' \\
2, & \mbox{if } a_k \mbox{ is assigned to } W''. 
\end{cases} 
$$
Note that $\mathrm{ind}(a_kb_k)\neq \mathrm{ind}(a_{k+1}b_{k+1})$ if and only if the pair $(a_kb_k, a_{k+1}b_{k+1})$ is an inversion. By definition, $C$ is direct if and only if $W'$ is a closed walk, i.e., if $\mbox{ind}(a_rb_r)=\mbox{ind}(a_1b_1),$ which is equivalent to saying that the number of switches between 1 and 2 in the sequence $$\mathrm{ind}(a_1b_1), \mathrm{ind}(a_2b_2), \ldots, \mathrm{ind}(a_rb_r)$$ is even.  
\end{proof}

\begin{example}
The nature of the two cycles in Example \ref{example:directtwisted} is now easily computed: the first contains no inversions and so is direct; the second contains one inversion $(23,12)$ and is therefore twisted.
\end{example}

The parity of a cycle in $G^{[2]}_{A,B}$ relates to the parity of its projection, as shown in the following result:

\begin{prop}\label{prop:parity}
Let $C$ denote a cycle in $G^{[2]}_{A,B}$.

(i) If $C$ is direct and $\pi(C)=\{W',W''\}$ then $P(C)=P(W')P(W'').$

(ii) If $C$ is twisted then $P(C)=-P(\pi(C)).$
\end{prop}
\begin{proof} Let $C=(e_1,\ldots, e_N)$ be such that the initial vertex of $e_1$ is an S-vertex,  and let $\{i_1,\ldots, i_M\}$, $\{j_1,\ldots, j_{N-M}\}$ be the partition of $\{1,\ldots N\}$ obtained as in Proposition~\ref{prop:proj}. In view of Remark \ref{rem:W'W''} we  have
\begin{equation}\label{eq:paritypf}
P(W')=(-1)^{\frac{M}{2}}\prod_{r=1}^{\frac{M}{2}}\left[ \mathrm{sign}(\pi(e_{i_{2r-1}}))\mathrm{sign}(\pi(e_{i_{2r-1}+1}))\right].
\end{equation}
(We adopt the convention that an empty product has value $1$.) Equation (\ref{eq:projsign}) implies that, for any $r\in \{1,\ldots, \frac{M}{2}\},$ 
$$\mathrm{sign}(\pi(e_{i_{2r-1}}))\mathrm{sign}(\pi(e_{i_{2r-1}+1}))=-\mathrm{sign}(e_{i_{2r-1}})\mathrm{sign}(e_{i_{2r-1}+1})$$
if the two S-vertices of the pair $e_{i_{2r-1}},e_{i_{2r-1}+1}$ form an inversion, and 
$$\mathrm{sign}(\pi(e_{i_{2r-1}}))\mathrm{sign}(\pi(e_{i_{2r-1}+1}))=\mathrm{sign}(e_{i_{2r-1}})\mathrm{sign}(e_{i_{2r-1}+1})$$
otherwise. Denoting by $\mathrm{inv}(W')$ the number of pairs from $(e_{i_1},e_{i_1+1}),(e_{i_3},e_{i_3+1}),\ldots, (e_{i_{M-1}},e_{i_M})$ whose S-vertices form an inversion, equation (\ref{eq:paritypf}) becomes
\begin{equation}\label{paritypf2}
P(W')=(-1)^{\frac{M}{2}}(-1)^{\mathrm{inv}(W')}\prod_{r=1}^{\frac{M}{2}}\mathrm{sign}(e_{i_{2r-1}})\mathrm{sign}(e_{i_{2r-1}+1}).
\end{equation}  
Multiplying with the corresponding equation for $W''$ yields
\[
P(W')P(W'')=(-1)^{\frac{M}{2}}(-1)^{\mathrm{inv}(W')}(-1)^{\frac{N-M}{2}}(-1)^{\mathrm{inv}(W'')} \prod_{r=1}^{N}\mathrm{sign}(e_r) =(-1)^{\mathrm{inv}(C)}P(C)
\]
where $\mathrm{inv}(C)=\mathrm{inv}(W')+\mathrm{inv}(W'')$ is the number of inversions in the S-vertex sequence of $C.$ It follows from Proposition \ref{prop:parityinv} that if $C$ is direct then $\mathrm{inv}(C)$ is even and part (i) of the conclusion follows. If $C$ is twisted then $\mathrm{inv}(C)$ is odd, $P(\pi(C))=P(W'\sqcup W'')=P(W')P(W'')$ and (ii) follows. 
\end{proof}

Having dealt in Proposition~\ref{prop:parity} with the behaviour of parity under projection we next ask what can be inferred about s-cycles in $G_{A,B}^{[2]}$ from their projections. 

\begin{prop}\label{prop:stoichproj}
Let $C$ be a cycle of $G_{A,B}^{[2]}$. 

(i) If $C$ is twisted then $\pi(C)$ is an s-walk if and only if $C$ is an s-cycle.

(ii) If $C$ is direct, $\pi(C)=\{W',W''\}$, and both $W'$ and $W''$ are s-walks, then $C$ is an s-cycle.
\end{prop}
\begin{proof} We may assume that $C$ is directed; otherwise, if all its edges are undirected, we simply choose one of the two possible orientations. By definition, $C$ is an s-cycle if 
\begin{equation}\label{eq:SR=RS}
\prod_{e\in\{\text{S-to-R edges of }C\}}l(e)=\prod_{e\in\{\text{R-to-S edges of }C\}}l(e).
\end{equation}
Since an S-to-R edge of $G_{A,B}^{[2]}$ projects to an S-to-R edge of $G_{A,B}$ with the same label, one has
$$\prod_{e\in\{\text{S-to-R edges of }C\}}l(e)=\prod_{e\in\{\text{S-to-R edges of }\pi(C)\}}l(e)$$ in case (i), and 
$$\prod_{e\in\{\text{S-to-R edges of }C\}}l(e)=\left(\prod_{e\in\{\text{S-to-R edges of }W'\}} l(e) \right)\left(\prod_{e\in\{\text{S-to-R edges of }W''\}}l(e)\right)$$
in case (ii). Similar equalities hold for R-to-S edges and the conclusion is now immediate from (\ref{eq:SR=RS}).
\end{proof}

In conjunction with Lemma \ref{lem:swalks}, Proposition \ref{prop:stoichproj} has an immediate notable consequence:
\begin{cor}\label{cor:scycles[2]}
If $G_{A,B}$ is steady, then the same is true for $G_{A,B}^{[2]}.$ In particular, if $G_{A,B}$ is acyclic, then $G_{A,B}^{[2]}$ is steady.
\end{cor}

\begin{remark}
\label{rem:allscycles}
The condition that $G_{A,B}$ is steady is often fulfilled in applications to CRNs. Corollary \ref{cor:scycles[2]} shows that when this is the case, verifying that $G_{A,B}^{[2]}$ is odd$^*$ boils down to ruling out odd intersections of e-cycles. 
\end{remark}

\end{subsection}

\begin{subsection}{Liftings of DSR cycles} 

As explained in Remark \ref{rem:pi}, the projection map $\pi$ is not injective;  in particular different cycles in $G^{[2]}_{A, B}$ may project to the same closed walk in $G_{A, B}$. If $W$ is a cycle of $G_{A, B}$ then any cycle $C$ of $G^{[2]}_{A, B}$ that projects to $W$ will be called a {\bf lifting} of $W.$ By way of Definition \ref{def:projcycle} and Remark \ref{rem:extendpi}, either $C$ is twisted or, if $\pi(C)=\{W',W''\}$ then $W'=\varnothing$ or $W''=\varnothing$. In the first case $C$ will be called an {\bf internal lifting} of $W$ and in the second, $C$ is an {\bf external lifting} of $W.$ The next proposition sheds light on this terminology. 

\begin{prop}
\label{prop:cyclelift}
Let $W$ be a cycle in $G_{A, B}$ with vertex sequence $(S_{i_1},R_{j_1},\ldots, S_{i_N}, R_{j_N}, S_{i_1})$.

(i) A cycle $C$ in $G^{[2]}_{A, B}$ is an external lifting of $W$ if and only if there exists $p \not\in\{i_1,\ldots, i_N\}$ such that
\[
C= (i_1p,j_1^p,\ldots,i_Np, j_N^p, i_1p).
\]

(ii) If $C$ is an internal lifting of $W$ then the S-vertices of $C$ are a subset of $\{i_1,\ldots, i_N\}\times\{i_1,\ldots, i_N\}.$
\end{prop}
\begin{proof}
(i) In one direction we can easily check by computing $\pi(C)$ that a cycle of the form $C$ is an external lifting of $W$. On the other hand, suppose an external lifting of $W$ gives rise to a cycle whose vertex sequence is not of this form, implying that $C$ contains three consecutive S-vertices $(ij, ij', i'j'),$ with no common element (namely $i,i',j,j'$ are all distinct). It follows from the proof of Proposition \ref{prop:proj} that each of $W'$ and $W''$ must include a vertex from $\{S_i,S_j\}$ and a vertex from $\{S_{i'},S_{j'}\},$ which means that each of $W'$ and $W''$ contains at least two S-vertices and therefore at least two edges. This contradicts the fact that one of them must be empty. 

(ii) Since $C$ is twisted, $W = W' \sqcup W''$. If $ij$ is an S-vertex of $C$ then, by construction, either $S_i$ is an S-vertex of $W'$ and $S_j$ is an S-vertex of $W''$, or vice versa. The conclusion follows.    
\end{proof}

\begin{remark}
\label{rem:liftparity}
It is clear from Proposition~\ref{prop:parity} that if a cycle $C$ in $G^{[2]}_{A, B}$ is an external lifting of a cycle $W$ in $G_{A, B}$, then $P(C) = P(W)$, while if $C$ is an internal lifting of $W$, then $P(C) = -P(W)$.  In other words internal liftings change parity, while external liftings preserve parity of cycles in $G_{A, B}$.
\end{remark}

\begin{remark}
\label{rem:shortcycles}
A cycle $C$ in $G^{[2]}_{A, B}$ of length $4$ must be of the form $(ap, c^p, dp, e^p, ap)$, with $a \ne d$ and $c \ne e$. Thus, by observation $\pi(C)$ is $(S_a, R_c, S_d, R_e, S_a)$, namely a cycle of length $4$ in $G_{A, B}$. As $C$ is clearly direct, it has the same parity as $\pi(C)$. Thus by Proposition~\ref{prop:cyclelift} cycles of length $4$ in $G^{[2]}_{A, B}$ arise precisely as external liftings of cycles of length $4$ in $G_{A, B}$. The same remark applies to cycles of length $2$, although these cannot occur in $G_{A, B}$ or $G^{[2]}_{A, B}$ if $A_{ij}B_{ji} \geq 0$.
\end{remark}

\begin{example}[Example~\ref{example_motivate1} revisited]
We may now argue that $G_{A,I}^{[2]}$ in this example is odd directly from $G_{A,I}$. Consider a cycle $E$ of $G_{A,I}^{[2]},$ and let $W'$ and $W''$ denote the two walks in $G_{A,I}$ that form its projection. Since there are ${4\choose 2}=6$ S-vertices in $G_{A,I}^{[2]},$  the length of $E$ can not exceed $12$. Notice that $G_{A,I}$ has two cycles: an even Hamiltonian cycle $C$, and an odd 4-cycle $D$. In view of Remark \ref{rem:W'W''}(2), neither $W'$ nor $W''$ can traverse back and forward consecutively an unoriented edge of $G_{A,I}$, and $\pi(E)$ must therefore consist of copies of $C$ and $D$. In fact, there is at most one copy of each of $C$ and $D$ in this union: two copies of $C$ would imply the contradiction $\mathrm{length}(E) \ge 16,$ while two copies of $D$ would imply that $E$ contains both liftings of $D$ allowed by Remark~\ref{rem:shortcycles} (where $p$ is $3$ and $4$ respectively), so $E$ cannot be a cycle itself. Now suppose that $W'\cup W''$ has the same edges as $C\cup D$ (counted  with multiplicity). Then the directed path $(S_2, R_2, S_1)$ is traversed twice, and therefore both directed paths $P_1 = (23, 2^3, 13)$ and $P_2 = (24, 2^4, 14)$ must exist in $E$. Further, $(S_4, R_4, S_3)$ must arise as the projection of one of the directed paths $P_3 = (14, 4^1, 13)$ or $P_4 = (24, 4^2, 23)$. However, clearly $P_3$ and $P_1$ cannot both exist in $E$ and similarly $P_4$ and $P_2$ cannot both exist in $E$, a contradiction. We conclude that $E$ is either a lifting of $C$ or a lifting of $D$. Since all liftings of e-cycle $D$ are internal and all liftings of o-cycle $C$ are external, it follows, by Remark~\ref{rem:liftparity}, that  $G_{A,I}^{[2]}$ is odd.
\end{example}

\begin{remark} 
\label{rem:multilift}
Although this could not occur in the previous example, in general the same cycle in $G_{A, B}$ may lift to multiple cycles of different parities in $G^{[2]}_{A, B}$, by lifting both externally and internally. This is illustrated in the following example (edge labels have been omitted):
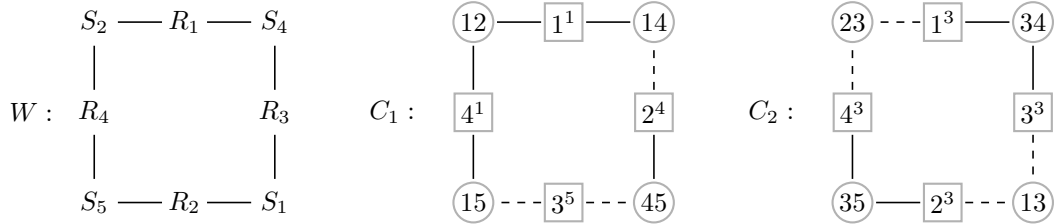
\begin{center}
\begin{tikzpicture}
[scale=1.2, place/.style={circle,draw=black!30,fill=black!0,thick,inner sep=0pt,minimum size=5.3mm},
transition/.style={rectangle,draw=black!30,fill=black!0,thick,inner sep=0pt,minimum size=5mm},
place1/.style={circle,draw=white!5,fill=white!9,thick,inner sep=0pt,minimum size=5.3mm},
transition1/.style={rectangle,draw=black!0,fill=black!0,thick,inner sep=0pt,minimum size=5mm},
pre/.style={->,shorten <=1pt,shorten >=1pt,>=stealth',semithick},
pres/.style={->,shorten <=-1pt,shorten >=-1pt,>=stealth',semithick},
pre1/.style={<-,shorten <=1pt,shorten >=1pt,>=stealth',semithick},
pres1/.style={<-,shorten <=-1pt,shorten >=0pt,>=stealth',semithick}];
\node at (-1.9,0) {$C_1:$};
\node[place] (12) at (-1,1) {$12$};
\node[place] (23) at (1,1) {$14$};
\node[place] (25) at (-1,-1) {$15$};
\node[place] (24) at (1,-1) {$45$};
\node[transition] ([12])  at (0,1)  {$1^1$}
edge[pre,-] (12)
edge [pre,-] (23);
\node[transition] ([22])  at (1,0)  {$2^4$}
edge[pre,-,dashed] (23)
edge [pre,-]  (24);
\node[transition] ([32])  at (0,-1)  {$3^5$}
edge[pre,-,dashed]  (24)
edge [pre,-,dashed] (25);
\node[transition] ([42])  at (-1,0)  {$4^1$}
edge[pre,-] (25)
edge [pre,-]  (12);

\begin{scope}[xshift=4.2cm]
\node at (-1.9,0) {$C_2:$};
\node[place] (12) at (-1,1) {$23$};
\node[place] (23) at (1,1) {$34$};
\node[place] (25) at (-1,-1) {$35$};
\node[place] (24) at (1,-1) {$13$};
\node[transition] ([12])  at (0,1)  {$1^3$}
edge[pre,-,dashed]  (12)
edge [pre,-]  (23);
\node[transition] ([22])  at (1,0)  {$3^3$}
edge[pre,-] (23)
edge [pre,-,dashed] (24);
\node[transition] ([32])  at (0,-1)  {$2^3$}
edge[pre,-,dashed] (24)
edge [pre,-] (25);
\node[transition] ([42])  at (-1,0)  {$4^3$}
edge[pre,-] (25)
edge [pre,-,dashed] (12);
\end{scope}

\begin{scope}[xshift=-4.2cm]
\node at (-1.7,0) {$W:$};
\node[place1] (12) at (-1,1) {$S_2$};
\node[place1] (23) at (1,1) {$S_4$};
\node[place1] (25) at (-1,-1) {$S_5$};
\node[place1] (24) at (1,-1) {$S_1$};
\node[transition1] ([12])  at (0,1)  {$R_1$}
edge[pre,-]  (12)
edge [pre,-]  (23);
\node[transition1] ([22])  at (1,0)  {$R_3$}
edge[pre,-]  (23)
edge [pre,-] (24);
\node[transition1] ([32])  at (0,-1)  {$R_2$}
edge[pre,-]  (24)
edge [pre,-]  (25);
\node[transition1] ([42])  at (-1,0)  {$R_4$}
edge[pre,-]  (25)
edge [pre,-]  (12);
\end{scope}
\end{tikzpicture}
\end{center}

$W$ is an e-cycle, as is its external lifting $C_2$ (note that it has four inversions). On the other hand, the internal lifting $C_1$ is an o-cycle with one inversion.

\end{remark}
\end{subsection}

\end{section}

\begin{section}{Further applications and limitations of the DSR$^{[2]}$ graph}
\label{sec:final}

We present some further results which illustrate use of the theory developed so far. In particular we show how it is possible in some cases, via the theory in Section~\ref{sec:theory}, to make claims about $(AB)^{[2]}$ by examination of the DSR graph $G_{A,B}$ and without explicit construction of $G^{[2]}_{A, B}$. 

\begin{subsection}{When the DSR graph has three S-vertices or three R-vertices}
\label{sec3sp}

Given $C \in \mathbb{R}^{3 \times m}$ and $A,B^t \in \mathcal{Q}_0(C)$, then $G^{[2]}_{A, B}$ has a particularly nice structure, illustrated {\em generically} in the figure below: 
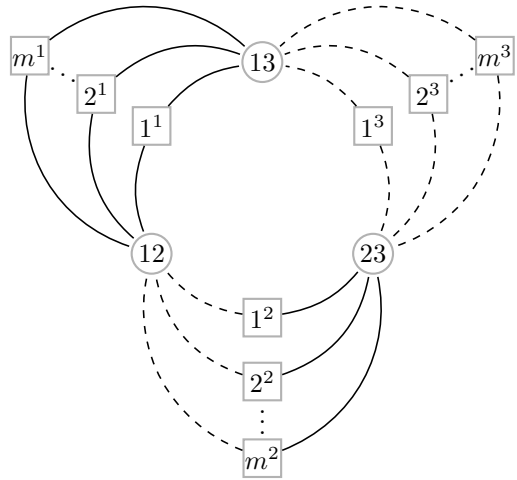
\begin{center}
\begin{tikzpicture}
[scale=.85, place/.style={circle,draw=black!30,fill=black!0 ,thick,inner sep=0pt,minimum size=5.3mm},
transition/.style={rectangle,draw=black!30,fill=black!0,thick,inner sep=0pt,minimum size=5mm},
pre/.style={<,shorten <=2pt,>=stealth',semithick},
post/.style={shorten >=1pt,>=stealth',semithick}]
\node[place] (12) at (-1.732,-1) {12};
\node[place] (23) at (1.732,-1) {23};
\node[place] (13) at (0, 2) {13};
\node[transition] ([13])  at (1.732,1)  {$1^3$}
edge [dashed,bend right = 20, post] (13)
edge [dashed,bend left = 20, post] (23);
\node[transition] ([12])  at (0,- 2) {$1^2$}
edge [dashed,bend left = 20, post] (12)
edge [bend right = 20, post] (23);
\node[transition] ([11])  at (-1.732,1) {$1^1$}
edge [bend right = 20, post] (12)
edge [bend left = 20, post] (13);
\node[transition] ([23])  at (1.5*1.732,1.5*1)  {$2^3$}
edge [dashed,bend right = 30, post] (13)
edge [dashed,bend left = 30, post] (23);
\node[transition] ([22])  at (1.5*0,- 1.5*2) {$2^2$}
edge [dashed,bend left = 30, post] (12)
edge [bend right = 30, post] (23);
\node[transition] ([21])  at (-1.5*1.732,1.5*1) {$2^1$}
edge [bend right = 30, post] (12)
edge [bend left = 30, post] (13);
\node[rotate=30] at (1.8*1.732,1.8*1) {$\ldots$};
\node at (1.8*0,- 1.8*1.95) {$\vdots$};
\node[rotate=-30] at (-1.8*1.732,1.8*1) {$\ldots$};
\node[transition] ([33])  at (2.1*1.732,2.1*1)  {$m^3$}
edge [dashed,bend right = 40, post] (13)
edge [dashed,bend left = 40, post] (23);
\node[transition] ([32])  at (2.1*0,- 2.1*2) {$m^2$}
edge [dashed,bend left = 40, post] (12)
edge [bend right = 40, post] (23);
\node[transition] ([31])  at (-2.1*1.732,2.1*1) {$m^1$}
edge [bend right = 40, post] (12)
edge [bend left = 40, post] (13);
\end{tikzpicture}
\end{center}
Depending on the particular entries of $A$ and $B,$ some edges may be absent, some edges may become directed, some edges may change sign, and some edges may get the label $\infty.$ Nonetheless, the generic structure of $G^{[2]}_{A, B}$ is always the same and in particular, the degree of any R-vertex is at most 2. This fact implies that no odd intersection of cycles is possible (see Remark~\ref{remoddintersect}) and checking that $G^{[2]}_{A, B}$ is odd$^*$ amounts to showing that all its e-cycles are s-cycles. This has a pleasant translation into a
property of the DSR graph, as follows:

\begin{thm}\label{thm:3sp}
Let $C\in \mathbb{R}^{3 \times m}$ and $A,B^t \in \mathcal{Q}_0(C)$. If $G_{A, B}$ satisfies (i) all e-cycles of length 4 are s-cycles; and (ii) all o-cycles of length 6 are s-cycles, then $G^{[2]}_{A,B}$ is odd$^*$, $(AB)^{[2]}$ is a $P_0$-matrix, and the nonreal spectrum of $AB$ does not intersect $\mathbb{C}_{-}$.
\end{thm}
\begin{proof}
If $G^{[2]}_{A,B}$ is odd$^*$, then the remaining conclusions follow from Theorem~\ref{mainP0thm} and Lemma~\ref{J2P0}. Since R-vertices of $G^{[2]}_{A,B}$ have degree at most 2, in order to show that $G^{[2]}_{A,B}$ is odd$^*$ we need only show that all e-cycles of $G^{[2]}_{A,B}$ are s-cycles. As explained in Remark \ref{rem:shortcycles}, all 4-cycles of $G^{[2]}_{A,B}$ are external liftings of 4-cycles of $G_{A, B}$, whose parity (and edge-labels) they inherit. It follows from (i) that all e-cycles of length 4 in $G^{[2]}_{A,B}$ are s-cycles. On the other hand, any 6-cycle $C$ of $G^{[2]}_{A,B}$ includes all S-vertices of $G_{A,B}^{[2]}$, namely $12, 23$ and $13$, and therefore has exactly one inversion. Thus $C$ is twisted (Proposition~\ref{prop:parityinv}), $\pi(C)$ is a closed walk, and moreover (Propositions~\ref{prop:parity}~and~\ref{prop:stoichproj}) $C$ is an e-cycle if and only if $\pi(C)$ is odd, and $C$ is an s-cycle if and only if $\pi(C)$ is an s-walk. 

We may distinguish three possibilities for a 6-cycle $C$, according to the three possible structures of its projection, illustrated below. Here vertices depicted as distinct are distinct, and the projected closed walks are obtained by following the arrows. As the hypothesis that $A,B^t \in \mathcal{Q}_0(C)$ means that no pair of vertices in $G_{A, B}$ are connected by more than one edge, two arrows between a pair of vertices mean that an undirected edge is traversed twice:
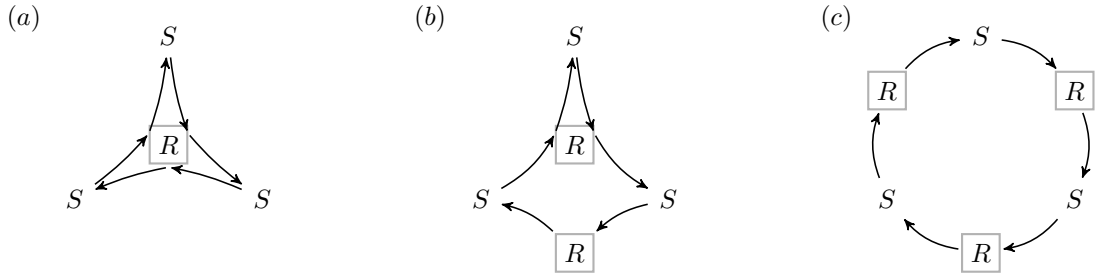
\begin{center}
\begin{tikzpicture}
[scale=1.2, place/.style={circle,draw=blue!50,fill=blue!5,thick,inner sep=0pt,minimum size=5.5mm},
transition/.style={rectangle,draw=black!50,fill=black!5,thick,inner sep=0pt,minimum size=5mm},
pre/.style={->,shorten <=1pt,shorten >=1pt,>=stealth',semithick},
pres/.style={->,shorten <=-1pt,shorten >=-1pt,>=stealth',semithick},
pre1/.style={<-,shorten <=1pt,shorten >=1pt,>=stealth',semithick},
pres1/.style={<-,shorten <=-1pt,shorten >=0pt,>=stealth',semithick}];

\begin{scope}[xshift=4.5cm]
\node at (-1.6,1.4) {$(c)$};
\node (R1) [transition, draw=black!30, fill=black!0] at (30:1.2) {$R$};
\node (R2) [transition, draw=black!30, fill=black!0] at (-90:1.2){$R$};
\node (R3) [transition, draw=black!30, fill=black!0] at (150:1.2) {$R$};
\node (12) at (-30:1.2) {$S$};
\node (23) at (-150:1.2) {$S$};
\node (25) at (90:1.2) {$S$};
\draw[->, >=stealth', semithick,shorten <=1pt,shorten >=1pt,bend left=20] (R1) to (12);
\draw[->, >=stealth', semithick,shorten <=1pt,shorten >=1pt,bend left=20] (12) to(R2);
\draw[->, shorten <=1pt,shorten >=1pt,>=stealth', semithick,bend left=20] (R2) to (23);
\draw[->, shorten <=1pt,shorten >=1pt,>=stealth', semithick,bend left=20] (23) to (R3);
\draw[->, shorten <=1pt,shorten >=1pt,>=stealth', semithick,bend left=20] (R3) to (25);
\draw[->, shorten <=1pt,shorten >=1pt,>=stealth', semithick,bend left=20] (25) to (R1);
\end{scope}

\begin{scope}[xshift=0cm]
\node at (-1.6,1.4) {$(b)$};
\node[transition, draw=black!30, fill=black!0] at (0,0) {$R$};
\node (R2) [transition, draw=black!30, fill=black!0] at (-90:1.2) {$R$};
\node (12) at (-30:1.2) {$S$};
\node (23) at (-150:1.2) {$S$};
\node (25) at (90:1.2) {$S$};
\draw[->, >=stealth', semithick,shorten <=1pt,shorten >=1pt, bend right=6] (25) to (30:.25);
\draw[->, >=stealth', semithick,shorten <=1pt,shorten >=1pt,bend right=15] (30:.25) to(12);
\draw[->, shorten <=1pt,shorten >=1pt,>=stealth', semithick,bend right=15] (12) to (R2);
\draw[->, shorten <=1pt,shorten >=1pt,>=stealth', semithick,bend right=15] (R2) to (23);
\draw[->, shorten <=1pt,shorten >=1pt,>=stealth', semithick,bend right=15] (23) to (150:.25);
\draw[->, shorten <=1pt,shorten >=1pt,>=stealth', semithick,bend right=6] (150:.25) to (25);
\end{scope}

\begin{scope}[xshift=-4.5cm]
\node at (-1.6,1.4) {$(a)$};
\node[transition, draw=black!30, fill=black!0] at (0,0) {$R$};
\node (12) at (-30:1.2) {$S$};
\node (23) at (-150:1.2) {$S$};
\node (25) at (90:1.2) {$S$};
\draw[->, >=stealth', semithick,shorten <=1pt,shorten >=1pt,bend right=6] (25) to(30:.25);
\draw[->, >=stealth', semithick,shorten <=1pt,shorten >=1pt,bend right=6] (30:.25) to(12);
\draw[->, shorten <=1pt,shorten >=1pt,>=stealth', semithick,bend right=6] (12) to (-90:.25);
\draw[->, shorten <=1pt,shorten >=1pt,>=stealth', semithick,bend right=6] (-90:.25) to (23);
\draw[->, shorten <=1pt,shorten >=1pt,>=stealth', semithick,bend right=6] (23) to (150:.25);
\draw[->, shorten <=1pt,shorten >=1pt,>=stealth', semithick,bend right=6] (150:.25) to (25);
\end{scope}
\end{tikzpicture}
\end{center}
We now need to show that the hypotheses of the theorem imply that odd walks in $G_{A, B}$ of types (a), (b) and (c) are closed s-walks. Type (a) is an acyclic walk and is structurally a closed s-walk (see Corollary~\ref{cor:tree-s}). Type (b) is composed of an undirected edge traversed twice and a 4-cycle. If the walk is odd then the 4-cycle must be even and thus an s-cycle by hypothesis (i); it follows that the walk itself is a closed s-walk. Finally, type (c) is covered by hypothesis (ii).
\end{proof}

As we might expect, there is a dual result to Theorem~\ref{thm:3sp} for matrices with three columns:
\begin{cor}\label{thm:3reac}
Let $C\in \mathbb{R}^{m \times 3}$ and $A,B^t \in \mathcal{Q}_0(C)$. If $G_{A, B}$ satisfies (i) all e-cycles of length 4 are s-cycles and (ii) all o-cycles of length 6 are s-cycles, then the nonreal spectrum of $AB$ does not intersect $\mathbb{C}_{-}$.
\end{cor}
\begin{proof}
Observe that $G_{A, B}$ satisfies conditions (i) and (ii) on its cycles if and only if $G_{A^t, B^t}$ does. By Theorem~\ref{thm:3sp}, the nonreal spectrum of $A^tB^t$ does not intersect $\mathbb{C}_{-}$. Since $AB$ has the same nonzero spectrum as $A^tB^t$ (Lemma~\ref{spectranposelemma}), the result follows. 
\end{proof}

\begin{remark}
The conclusions in Example~\ref{ex3row} and Proposition~\ref{3sp1} are an immediate consequence of Theorem~\ref{thm:3sp}. 
\end{remark}

\begin{subsection}{When the DSR graph is steady with a vertex set of low degree}

Whereas the results of Section~\ref{sec3sp} dealt with matrices restricted by size, we now consider matrices of arbitrary dimension but with restrictions on the number of nonzero entries in rows/columns. The following generalisation of Proposition~\ref{prop:switchRS} is immediate:

\begin{thm}
\label{thm:alls}
Let $A \in \mathbb{R}^{n \times m}$ be such that $G_{A,A^t}$ is steady. 
\begin{itemize}
\item[(i)] If $A$ has no more than two nonzero entries in each column then, for any $B \in \mathcal{Q}_0(A^t)$, $AB$ and $(AB)^{[2]}$ are $P_0$-matrices and so $AB$ is positive semistable. 
\item[(ii)] If $A$ has no more than two nonzero entries in each row then, for any $B \in \mathcal{Q}_0(A^t)$, $AB$ is positive semistable.
\end{itemize}
\end{thm}
\begin{proof}
$G_{A,A^t}$ is steady by assumption and $G_{A,A^t}^{[2]}$ is steady by Corollary \ref{cor:scycles[2]}. The proof now proceeds exactly as that of Proposition~\ref{prop:switchRS}.
\end{proof}

\end{subsection}

\begin{subsection}{When the DSR graph is acyclic}

\begin{thm}
\label{thm:tree}
Let $C \in \mathbb{R}^{n \times m}$ be such that $G_{C,C^t}$ is acyclic. Then, for each $A \in \mathcal{Q}_0(C)$ and $B \in \mathcal{Q}_0(C^t)$, $AB$ and $(AB)^{[2]}$ are $P_0$-matrices and $AB$ is positive semistable. 
\end{thm}

\begin{proof}
Since $G_{C,C^t}$ is acyclic, it is odd, and so by Theorem~\ref{mainP0thm} $AB$ is a $P_0$-matrix for each $A, B^t \in \mathcal{Q}_0(C)$. Observe that (i) given any $A \in \mathcal{Q}(C)$, since $G_{A,A^t}$ is acyclic, $G^{[2]}_{A, A^t} = G_{\overline {\mathbf{L}}^A, (\overline {\mathbf{L}}^A)^t}$ is steady by Corollary~\ref{cor:scycles[2]} and (ii) given any $A,B^t \in \mathcal{Q}(C)$, $\overline {\mathbf{L}}^A \in \mathcal{Q}(\overline {\mathbf{L}}^C)$ and $\underline {\mathbf{L}}^B \in \mathcal{Q}((\overline {\mathbf{L}}^C)^t)$. By Lemma~\ref{lem:scyclePmat} (i) and (ii) imply that $(AB)^{[2]} = \overline {\mathbf{L}}^A\underline {\mathbf{L}}^B$ is a $P_0$-matrix for each $A, B^t \in \mathcal{Q}(C)$ and the result extends by closure to $A, B^t \in \mathcal{Q}_0(C)$. Since for each $A, B^t \in \mathcal{Q}_0(C)$ both $AB$ and $(AB)^{[2]}$ are $P_0$-matrices, the conclusion that $AB$ is positive semistable now follows from Lemma~\ref{lemposstable0}.
\end{proof}

\begin{remark}
\label{rem:lemcolun}
Lemma~\ref{lem:columnvec} is clearly a special case of Theorem~\ref{thm:tree}. If $C$ is a column vector then $G_{C, C^t}$ is in fact a star. 
\end{remark}

\begin{remark}
\label{rem:notoddstar}
If $G_{C,C^t}$ is acyclic, this does not imply that $G^{[2]}_{C,C^t}$ is odd$^*$. While most of the examples presented have involved situations where the DSR$^{[2]}$ graph is odd or odd$^*$, Theorem~\ref{thm:tree} illustrates that the DSR$^{[2]}$ graph may provide useful information even in the case where it fails to be odd$^*$. 
\end{remark}

\begin{remark}
\label{rem:notacyclic}
Observe that Theorem~\ref{thm:tree} does not imply that if $G_{A,B}$ is acyclic, then $AB$ is positive semistable. However, given $A,B^t \in \mathbb{R}^{n \times m}$, there exists $C\in \mathbb{R}^{n \times m}$ with $G_{C,C^t}$ acyclic and such that $A \in \mathcal{Q}_0(C)$, $B \in \mathcal{Q}_0(C^t)$ if and only if the underlying undirected graph of $G_{A,B}$ is acyclic (i.e., $G_{A,B}$ has no semicycles). Theorem~\ref{thm:tree} could thus be rephrased: ``If the underlying undirected graph of $G_{A,B}$ is acyclic, then $AB$ and $(AB)^{[2]}$ are $P_0$-matrices and $AB$ is positive semistable.''
\end{remark}

\end{subsection}

\subsection{Limitations of the theory developed so far}

There are a number of ways in which naive application of the theory described here can fail to provide information, even when a set of matrices structurally forbids the passage of a pair of nonreal eigenvalues through the imaginary axis. We illustrate with a couple of examples. Following up on Remark~\ref{rem:notoddstar}, the next example illustrates that even when $G^{[2]}_{A,B}$ is not odd$^*$ and in fact $(AB)^{[2]}$ is not necessarily a $P_0$-matrix, it may still be possible to infer the absence of Hopf bifurcation from the DSR$^{[2]}$ graph:

\begin{example}
\label{exmatchings}
Consider the matrices
\[
A=\left(\begin{array}{ccc}a&0&0\\0&b&0\\0&0&c\end{array}\right)\,, \quad B=\left(\begin{array}{ccc}0&-d&0\\-e&0&f\\g&0&h\end{array}\right)\,,
\]
with $a,b,c,d,e,f,g,h > 0$. $G^{[2]}_{A, B}$ is shown:
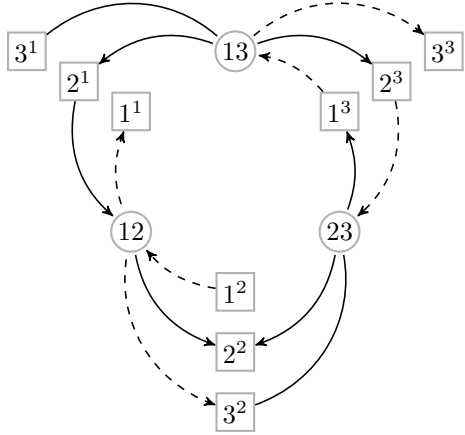
\begin{center}
\begin{tikzpicture}
[scale=0.8, place/.style={circle,draw=black!30,fill=black!0 ,thick,inner sep=0pt,minimum size=5.3mm},
transition/.style={rectangle,draw=black!30,fill=black!0,thick,inner sep=0pt,minimum size=5mm},
pre/.style={<,shorten <=2pt,>=stealth',semithick},
post/.style={shorten >=1pt,>=stealth',semithick}]
\node[place] (12) at (-1.732,-1) {$12$};
\node[place] (23) at (1.732,-1) {$23$};
\node[place] (13) at (0, 2) {$13$};
\node[transition] ([13])  at (1.732,1)  {$1^3$}
edge [->, dashed,bend right = 20, post]  (13)
edge [<-, bend left = 20, post]  (23);
\node[transition] ([12])  at (0,- 2) {$1^2$}
edge [->, dashed,bend left = 20, post]  (12);

\node[transition] ([11])  at (-1.732,1) {$1^1$}
edge [<-, dashed, bend right = 20, post]  (12);

\node[transition] ([23])  at (1.5*1.732,1.5*1)  {$2^3$}
edge [<-, bend right = 30, post]  (13)
edge [->, dashed, bend left = 30, post]  (23);
\node[transition] ([22])  at (1.5*0,- 1.5*2) {$2^2$}
edge [<-, bend left = 30, post]  (12)
edge [<-, bend right = 30, post]  (23);
\node[transition] ([21])  at (-1.5*1.732,1.5*1) {$2^1$}
edge [->, bend right = 30, post]  (12)
edge [<-, bend left = 30, post]  (13);

\node[transition] ([33])  at (2*1.732,2*1)  {$3^3$}

edge [<-, dashed,bend right = 40, post] (13);
\node[transition] ([32])  at (2*0,- 2*2) {$3^2$}
edge [<-, dashed, bend left = 40, post] (12)
edge [bend right = 40, post] (23);
\node[transition] ([31])  at (-2*1.732,2*1) {$3^1$}

edge [bend left = 40, post] (13);
\end{tikzpicture}

\end{center}
Labels have been omitted, but it is easy to check that $G^{[2]}_{A, B}$ fails to be odd$^*$ as a consequence of the cycle $C = (13, 2^3, 23, 1^3, 13)$ which is even but, involving $\infty$ labels, not an s-cycle. Indeed $(AB)^{[2]}$ is not necessarily a $P_0$-matrix for all values of the entries. However computation reveals that $(AB)^{[2]}$ is nonsingular and applying Lemma~\ref{J2nonsingular}, $\mathrm{Spec}\,(AB)\backslash\{0\}$ does not intersect the imaginary axis, forbidding Hopf bifurcation. 

In fact, more careful application of the theory developed in \cite{banajicraciun2} allows us to infer from $G^{[2]}_{A, B}$, and without direct computation of $\mathrm{det}((AB)^{[2]})$, that $\mathrm{det}((AB)^{[2]}) > 0$: the argument is only sketched here. There are exactly two cycles in $G^{[2]}_{A, B}$, an odd cycle $D = (12, 3^2, 23, 1^3, 13, 2^1, 12)$ and the even cycle $C$. However $C$ does not contribute any negative terms to the determinant: this follows from the fact that no undirected edges are incident into S-vertex $12$ and consequently there is no set of three R-vertices matched with the three S-vertices in both $G^{[2]}_{A, 0}$ and $G^{[2]}_{0, B}$ in such a way that the union of these matchings includes $C$. See Remark~\ref{rem:P0proof}; see also the discussion on ``good'' cycles \cite{Vazirani} (resp. ``central'' cycles in \cite{RobertsonSeymourThomas}) and Pfaffian orientations. Since $C$ is the only even cycle in $G^{[2]}_{A, B}$ and contributes nothing to the determinant, $\mathrm{det}((AB)^{[2]}) \geq 0$. That $\mathrm{det}((AB)^{[2]}) \neq 0$ follows from the existence of $D$, which traverses each S-vertex and thus contributes a nonzero term to $\mathrm{det}((AB)^{[2]})$. 
\end{example} 

\begin{remark}
\label{rem:adddiag1}
Remark~\ref{rem:adddiag} does not apply to Example~\ref{exmatchings}, since $(AB)^{[2]}$ has not been proved to be a $P_0$-matrix. 
\end{remark}

The next example demonstrates that $(AB)^{[2]}$ may be a $P_0$-matrix for all $B \in \mathcal{Q}_0(A^t)$ although this does not follow from examination of $G^{[2]}_{A,A^t}$ using any of the theory presented so far. It seems likely that this information can indeed be inferred from $G^{[2]}_{A,A^t}$, but that new techniques will be needed. 

\begin{example}
\label{ex4by5}
Consider the matrix
\[
A=\left(\begin{array}{rrrr}1&0&0&0\\-1&1&0&0\\0&-1&1&0\\-1&0&-1&1\\0&0&0&-1\end{array}\right).
\]
Symbolic computation confirms that $(AB)^{[2]}$ is a $P_0$-matrix for arbitrary $B \in \mathcal{Q}_0(A^t)$ . Since $G_{A, A^t}$ is odd, $AB$ is also a $P_0$-matrix, and so, by Lemma~\ref{lemposstable0}, $AB$ is in fact positive semistable. However, neither $G^{[2]}_{A, A^t}$ nor $G^{[2]}_{A^t, A}$ is odd$^*$. The picture on the left shows two e-cycles of $G^{[2]}_{A, A^t}$ with odd intersection. On the right $G^{[2]}_{\tilde A^t, \tilde A}$ is depicted, where $\tilde A$ denotes the submatrix of $A$ obtained by removing its first and last rows. The e-cycles $(13, 2^1, 12, 3^2, 23, 1^3, 13)$ and $(13, 3^1, 14, 1^4, 24, 3^2, 23, 1^3, 13)$ have odd intersection, and $G^{[2]}_{\tilde A^t, \tilde A}$ fails to be odd$^*$.  By Lemma~\ref{lem:leafremove}, $G^{[2]}_{A^t, A}$ is not odd$^*$ either.

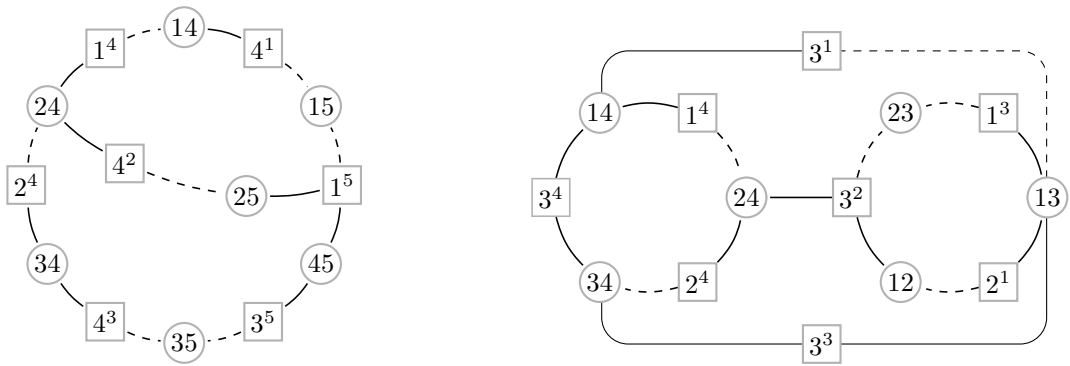
\begin{center}
\begin{tikzpicture}
[scale=2.1, place/.style={circle,draw=black!30,fill=black!0,thick,inner sep=0pt,minimum size=5.3mm},
transition/.style={rectangle,draw=black!30,fill=black!0,thick,inner sep=0pt,minimum size=5mm},
place1/.style={circle,draw=black!0,fill=black!0,thick,inner sep=0pt,minimum size=5.5mm},
transition1/.style={rectangle,draw=black!0,fill=black!0,thick,inner sep=0pt,minimum size=5mm},
pre/.style={shorten <=1pt,shorten >=1pt,>=stealth',semithick},
pres/.style={->,shorten <=-1pt,shorten >=-1pt,>=stealth',semithick},
pre1/.style={<-,shorten <=1pt,shorten >=1pt,>=stealth',semithick},
pres1/.style={<-,shorten <=-1pt,shorten >=0pt,>=stealth',semithick}];
\node[place] (15) at (30:1) {$15$};
\node[place] (14) at (90:1) {$14$};
\node[place] (24) at (150:1) {$24$};
\node[place] (34) at (210:1) {$34$};
\node[place] (35) at (270:1) {$35$};
\node[place] (45) at (330:1) {$45$};
\node[place] (25) at (-10:.4) {$25$};
\node[transition] ([42]) at (160:.4) {$4^2$}
edge[pre,bend right=7, dashed, semithick] (25)
edge[pre,bend left=7,semithick] (24);
\node[transition] ([15]) at (0:1) {$1^5$}
edge[pre,bend right=10, dashed,semithick] (15)
edge[pre,bend left=7,semithick] (25)
edge[pre,bend left=10,semithick] (45);
\node[transition] ([41]) at (60:1) {$4^1$}
edge[pre,bend right=10,semithick] (14)
edge[pre,bend left=10, dashed,semithick] (15);
\node[transition] ([14]) at (120:1) {$1^4$}
edge[pre,bend right=10,semithick] (24)
edge[pre,bend left=10, dashed,semithick] (14);
\node[transition] ([24]) at (180:1) {$2^4$}
edge[pre,bend right=10,semithick] (34)
edge[pre,bend left=10, dashed,semithick] (24);
\node[transition] ([43]) at (240:1) {$4^3$}
edge[pre,bend right=10, dashed,semithick] (35)
edge[pre,bend left=10,semithick] (34);
\node[transition] ([35]) at (300:1) {$3^5$}
edge[pre,bend right=10,semithick] (45)
edge[pre,bend left=10, dashed,semithick] (35);
\end{tikzpicture}
\hspace{2cm}
\begin{tikzpicture}
[scale=1.3, place/.style={circle,draw=black!30,fill=black!0 ,thick,inner sep=0pt,minimum size=5.3mm},
place1/.style={circle,draw=white!5,fill=white!9,thick,inner sep=0pt,minimum size=5.3mm},
transition1/.style={rectangle,draw=black!30,fill=black!0,thick,inner sep=0pt,minimum size=5mm},
transition/.style={rectangle,draw=black!30,fill=black!0,thick,inner sep=0pt,minimum size=5mm},
pre/.style={shorten >=2pt,>=stealth',semithick,dashed},
post/.style={shorten >=1pt,>=stealth',semithick}]
\draw[rounded corners=10pt, dashed] (1,1.5) -- (3.3,1.5)--(3.3,0);
\draw[rounded corners=10pt] (1,1.5) -- (-1.25,1.5)--(-1.25,.86);
\draw[rounded corners=10pt] (1,-1.5) -- (3.3,-1.5)--(3.3,0);
\draw[rounded corners=10pt] (1,-1.5) -- (-1.25,-1.5)--(-1.25,-.86);
\node[place,xshift=-1cm] (24) at (0:1) {$24$};
\node[place, xshift=-1cm] (14) at (120:1) {$14$};
\node[place, xshift=-1cm] (34) at (240:1) {$34$};
\node[place, xshift=3cm] (23) at (120:1) {$23$};
\node[place,xshift=3cm] (12) at (240:1) {$12$};
\node[place,xshift=3cm] (13) at (0:1) {$13$};
\node[transition,xshift=-1cm] ([14])  at (60:1)  {$1^4$}
edge [post, bend right=15,semithick] (14)
edge [pre,bend left=15,semithick] (24);
\node[transition,xshift=-1cm,semithick] ([34])  at (180:1)  {$3^4$}
edge [post, bend left=15,semithick] (14)
edge [post, bend right=15,semithick] (34);
\node[transition,xshift=-1cm] ([24])  at (300:1)  {$2^4$}
edge [post, bend right=15,semithick] (24)
edge [pre, bend left=15,semithick] (34);
\node[transition, xshift=3cm] ([32])  at (180:1)  {$3^2$}
edge [post, bend right=15,semithick] (12)
edge [post] (24)
edge [pre,bend left=15] (23);
\node[transition,xshift=3cm] ([21])  at (300:1)  {$2^1$}
edge [pre, bend left=15] (12)
edge [post, bend right=15] (13);
\node[transition,xshift=3cm] ([13])  at (60:1)  {$1^3$}
edge [post,bend left=15] (13)
edge [pre,bend right=15] (23);
\node[transition] ([31])  at (1,1.5)  {$3^1$};
\node[transition] ([33])  at (1,-1.5)  {$3^3$};
\end{tikzpicture}
\end{center}

Notice that although $G^{[2]}_{A, A^t}$ is steady, Lemma~\ref{lem:scyclePmat} cannot be applied since $G^{[2]}_{B^t, B}$ is not in general steady. An interesting feature of this example is that some products of the form $\overline{\mathbf{L}}^A[\alpha|\beta]\underline{\mathbf{L}}^B[\beta|\alpha]$ in (\ref{eqCB2}) contain negative terms, but all such negative terms are cancelled by positive terms arising from other products. 
\end{example} 

The last two examples demonstrate that DSR$^{[2]}$ graphs have more to offer than is explored in this paper. In particular a more complete analysis of projection and lifting may provide new techniques to rule out Hopf bifurcation in families of matrices. 

We conclude by remarking that criteria such as ``$(AB)^{[2]}$ is a $P_0$-matrix'' or even ``$\mathrm{det}((AB)^{[2]}) > 0$'' which forbid the passage of a real eigenvalue of $(AB)^{[2]}$ through zero are sufficient but not necessary to preclude Hopf bifurcation. Indeed a real eigenvalue of $(AB)^{[2]}$ may pass through zero as a consequence of migration of eigenvalues of $AB$ on the real axis without any possibility of Hopf bifurcation. This limitation may be circumvented by resorting to conditions that distinguish imaginary pairs of eigenvalues of $AB$ from pairs of real eigenvalues summing to zero, for example based on subresultant criteria \cite{Guckenheimer1997aa}. Natural extensions of the work here involve exploiting such criteria, and also examining other graphs related to $AB$ whose cycle structure contains information about Hopf bifurcation. 

\end{subsection}

\end{section}

\section*{Acknowledgements}
The authors acknowledge support from Leverhulme grant F/07 058/BU ``Structural conditions for oscillation in chemical reaction networks''. This work was completed during Casian Pantea's stay as a research associate in the Department of Electrical and Electronic Engineering, Imperial College London.

\bibliographystyle{siam}

\end{document}